\documentclass[a4paper,11pt]{article}
\usepackage{amsmath, amsthm, amssymb}
\usepackage{graphicx}
\usepackage{enumerate}

\newcommand{\R}{{\cal{R}}}
\renewcommand{\P}{{\cal{P}}}
\newcommand{\Sb}{{\cal{B}}}

\newcommand{\eps}{{\varepsilon}}
\newcommand{\lp}[1]{\left\|{#1}\right\|}
\newcommand{\pmX}{{X^{\pm 1}}}

\newcommand{\eqb}[1]{\mathop{=}_{#1}}

\newcommand{\cord}[2]{\left[#1\right]_{#2}}
\newcommand{\Comp}[1]{c\left(#1\right)}

{\catcode`\|=\active
  \gdef\Set#1{\left\{\:{\mathcode`\|"8000\let|\SetVert #1}\:\right\}}}
{\catcode`\|=\active
  \gdef\Pres#1{\left\langle\:{\mathcode`\|"8000\let|\SetVert #1}\:\right\rangle}}
\def\SetVert{\egroup\;\middle|\;\bgroup}

\DeclareMathOperator{\Grp}{Grp}
\DeclareMathOperator{\Sgp}{Sgp}

\newtheorem{theorem}{Theorem}
\newtheorem{corollary}[theorem]{Corollary}
\newtheorem{lemma}[theorem]{Lemma}

\newtheorem{proposition}[theorem]{Proposition}

\newtheorem{remark}[theorem]{Remark}

\newtheorem{observation}[theorem]{Observation}
\newtheorem{construction}[theorem]{Construction}
\newtheorem{question}[theorem]{Question}

\theoremstyle{definition}
\newtheorem{definition}[theorem]{Definition}

\title{From small-overlap conditions to automatic semigroups}
\author{Uri Weiss \\ uriw@tx.technion.ac.il}

\begin{document}

\maketitle

\begin{abstract}
We study the connection between small-overlap conditions and automaticity of semigroups. We restrict the discussion to conditions that imply embeddability and under which each relation decomposes into at least seven pieces. For these hyperbolic-like conditions we show how to construct an automatic structure. Furthermore, we show that the naive approach of considering just geodesics fails in our case.
\end{abstract}


\section{Introduction}

Considering semigroups from the combinatorial and geometric point of view is an active research field in recent years. A major theme in this line of thinking is the transfer of ideas from combinatorial and geometric group theory into the language of semigroup theory. For example, the definitions of hyperbolic groups and automatic groups were extended to semigroups; see \cite{CRRT01, DG04, HKOT01, HT05, HT06, Mea07, SS04}.

One source of difficulty comes from the structure of the (right) Cayley graph of the semigroup. In groups the Cayley graph is a homogeneous space and enjoys a natural metric which is known as the word metric. In semigroups this is no longer true; the Cayley graph is not homogeneous and it is not clear how to define a useful metric on it. We will therefore focus on the case where the semigroup is embeddable. In this case, one may use the metric induced from an embedding of a semigroup Cayley graph into a group Cayley graph. However, as we shall see, this alone doesn't allow transfer of results from groups to semigroups.

In \cite{Rem71} the idea of van Kampen diagrams is extended to the case of monoids and semigroups. The author there uses small-overlap conditions to solve the word problem and to prove Adjan's criterion for embeddability. In groups, small-cancellation conditions imply automaticity \cite{GS90,GS91,Wei07} and certainly hyperbolicity implies automaticity. In this work we will consider hyperbolic-like  small-overlap conditions that imply automaticity (but also embeddablity into a group). This will give a partial answer to a question asked in \cite{DG04}. 

Before we can state the main theorem we need some terminology. Let $\P = \Pres{X| L_1=R_1, \ldots, L_n=R_n}$ be a semigroup presentation. The set of relations in $\P$ is called the \emph{defining relations}; the set of words appearing in the defining relations is called the \emph{defining words} and we denote it by $\R = \R(\P) = \Set{L_1,\ldots,L_n} \cup \Set{R_1,\ldots,R_n}$. A piece for $\P$ is a word $P$ such that there are two defining words $W_1$ and $W_2$ which decompose as $W_1=U_1 P U_2$ and $W_2 = V_1 P V_2$, respectively, and either $U_1\neq V_1$ or $U_2\neq V_2$. For a word $W\in X^*$ we denote by $\lp{W}$ the \emph{piece-length} of $W$, namely, the minimal $k$ such that $W=P_1\cdots P_k$ and $P_1,\ldots, P_k$ are pieces (it is zero if $W$ is the empty word and it is $\infty$ if no such decomposition exists). Let $\P = \Pres{X| L_1=R_1, \ldots, L_n=R_n}$ be a semigroup presentation. We say that $\P$ is a $K_3^2$ presentation \cite{Kas92, Gub94} if the following conditions hold:
\begin{enumerate}
 \item[(a)] Each defining relation $L=R$ has the property that $L$ and $R$ both start (respectively, end) with different generators.
 \item[(b)] Each defining word $W$ has a piece-length of at least $3$ (i.e., $\lp{W}\geq 3$).
 \item[(c)] If $R_1=L_1$ and $R_2=L_2$ are two defining relations then all four words $R_1$, $R_2$, $L_2$, and $L_2$ are distinct.
\end{enumerate}
Condition $K_3^2$ implies \cite{Gub94} that the semigroup presented by $\P$ is embeddable. Our main theorem is the following:
\begin{theorem}[Main Theorem] \label{thm:mainThm}
Let $\P = \Pres{X| L_1=R_1, \ldots, L_n=R_n}$ be a semigroup presentation. Assume that the $K_3^2$ condition holds and also:
\begin{enumerate}
 \item[\emph{(\dag)}] Each defining relation $L=R$ has the property that $\lp{R}+\lp{L}\geq7$.
\end{enumerate}
Then, the semigroup presented by $\P$ is automatic.
\end{theorem}

Recently, and independently, Mark Kambites \cite{Kam09} has shown that semigroups for which each defining word has a piece-length at least four, also known as $C(4)$ semigroups, are asynchronously-automatic. This is a weaker notion than automaticity. However, the $C(4)$ small-overlap assumption is weaker than the assumptions in the main theorem since it does not imply that the semigroup is embeddable. The following is a natural question:

\begin{question}
Can one show that the $C(4)$ small-overlap condition implies that the semigroup is automatic? (or find a counter example.)
\end{question}  

$C(7)$ groups (in which each relator decomposes to at least seven pieces) are hyperbolic. Now, in hyperbolic groups one may construct an automatic structure by considering the set of geodesics. We give an example (Section \ref{sec:example}) showing that even in our restricted setup one cannot use the set of geodesics as an automatic structure.

Automatic semigroups (or monoids) are defined using language theoretic notions (see definition \ref{def:autoSG}). There are `geometric characterizations' \cite{HT06,SS04} which are not as simple (or nice) as in the group case. We show, however, that for embeddable semigroups one can use the geometric characterization (known as fellow-traveller property). This can be done by considering the metric on the Cayley graph which is induced from the embedding of the semigroup.

For semigroups, the conjugacy problem is the problem of deciding if for two elements $A$ and $B$ there are other elements $U$ and $W$ such that $AU=UB$ and $WA=BW$ (equality of elements in the semigroup). For bi-automatic groups this problem is decidable \cite[Thm. 2.5.7]{EPS92}. A straightforward generalization of the proof of the main theorem shows that in fact the semigroups considered in the main theorem are \emph{bi-automatic} (this is done in \cite{Wei10}). Thus, using the same proof as the one in \cite{EPS92} we get that the conjugacy problem for these semigroups is decidable. Note however that the complexity of the solution is doubly exponential.

The rest of this paper is organized as follows. In section \ref{sec:preliminaries} we give the basic definition and notations. In Section \ref{sec:autoInSemigrp} we give the characterization of automaticity for embeddable semigroups and we show how to prove automaticity using a special order on the elements of the semigroup. In section \ref{sec:example} we give an example of an embeddable semigroup for which the conditions of the main theorem hold but the set of geodesics is not an automatic structure. In Section \ref{sec:vanKampen} we recall the parts of van Kampen diagram theory we need for the proof. Finally, in Section \ref{sec:proofMainThm} we prove the main theorem.

This work is part of the author's Ph.D. research conducted under the supervision of Professor Arye Juh\a'{a}sz.

\section{Preliminaries} \label{sec:preliminaries}

The following notations and definitions are based on \cite{CRRT01}. Let $S$ be a semigroup finitely generated by $X$. We denote by $X^+$ the set of non-empty words with letters in $X$, i.e., this is the free semigroup over $X$. We denote by $\eps$ the empty word and by $X^*$ the set $X^+\cup\Set{\eps}$, which is the free monoid over $X$. Given a word $W$ in $X^+$ we denote by $\overline{W}$ the element that $W$ presents and denote by $\pi_{X,S}:X^+\to S$ the natural map such that $\pi_{X,S}(W)=\overline{W}$ for all $W\in X^+$. For the purpose of this work it is enough to consider only semigroups with $1$ (i.e., monoids). We will adopt this convention in the sequel and assume that the natural map $\pi$ send $\eps$ to $1$. We denote the length of $W$ by $|W|$. We say that $U$ is a \emph{subword} of $W$ if $W$ has a decomposition $W = V_1 U V_2$; $U$ is a \emph{prefix} of $W$ if $V_1 = \eps$ and it is a \emph{suffix} of $W$ if $V_2 = \eps$. If $W$ and $V$ in $X^*$ present the same element in $S$ (i.e., $\pi_{X,S}(W)=\pi_{X,S}(U)$) than we may emphasis that the equality is in $S$ by writing $W\eqb{S}U$. If $S$ and $X$ are understood from the context we may also simply write $\pi$ instead of $\pi_{X,S}$. When needed, we will distinguish between semigroup presentations and group presentation using the notations $\Sgp\Pres{\cdot|\cdot}$ and $\Grp\Pres{\cdot|\cdot}$, respectively. Suppose $\$\not\in X$. We denote by $X(2,\$)$ the set $((X\cup\Set{\$})\times (X\cup\Set{\$})) \setminus \Set{(\$,\$)}$. 

\begin{definition} \label{def:deltaX}
Let the map $\delta_X:X^*\times X^* \to X(2,\$)^*$ be defined on
$(W,U)$, for $W=x_1\,x_2\,\cdots\,x_n$ and $U=y_1\,y_2\,\cdots
y_m$ as follows:
\[
\delta_X(W,U) = \left\{
\begin{array}{ll}
  (x_1,y_1)\cdots(x_n,y_n)(\$,y_{n+1})\ldots(\$,y_{m}) & n<m \\
  (x_1,y_1)\cdots(x_m,y_m)(x_{m+1},\$)\ldots(x_{n},\$) & m<n \\
  (x_1,y_1)\cdots(x_n,y_m) & m=n
\end{array}
\right.
\]
\end{definition}

\begin{definition}\label{def:autoSG}
A finitely generated semigroup $S$ is \emph{automatic} if there is
a generating set $X$ and regular language $L\subseteq X^*$ such
that:
\begin{enumerate}
 \item $L$ is onto $S$ through the natural map.
 \item For any $x\in X\cup\Set{\eps}$ the following set
 is regular:
 \[
 L_{x}=\Set{(W,U)\delta_X|W,U\in L;\; \overline{Wx}=\overline{U}}
 \]
\end{enumerate}
A language $L$ having these properties is called an
\emph{automatic structure} of $S$.
\end{definition}

We will be using \emph{shortlex} ordering in several places. Suppose we are given two vectors $v=(a_1,\ldots,a_n)$ and $u=(b_1,\ldots,b_m)$ with entries in some set $A$. If we wish to compare between $u$ and $v$ using a shortlex order then there will be some (complete) order on $A$ and, based on that order, $u$ precedes $v$ if:
\begin{enumerate}
	\item $n<m$; or,
	\item $n=m$ and there is some index $1\leq k \leq n$ such that $a_i=b_i$ for all $1\leq i \leq k-1$ and $a_k < b_k$.
\end{enumerate}
In some cases there may be a natural order defined on $A$ (e.g., if $A$ consists of natural numbers). If that is the case then the shortlex ordering (on vectors with entries in $A$) will be based on the natural order on $A$. We may also use the shortlex ordering to compare between words in $X^*$. In this case there will be a fixed (arbitrary) order on $X$ and we would consider the elements of $X^*$ as vectors with entries in $X$. The shortlex ordering is denoted by `$<_{\text{lex}}$'.

Suppose we are given a semigroup $S$ and a (finite) generating set $X$. The Cayley graph $\Gamma(S,X)$ of $S$ under the generating set $X$ (often abbreviated as $\Gamma$ if $S$ and $X$ are understood from the context) is the graph with $S$ as the vertex set and an edge $s \stackrel{x}{\rightarrow} sx$ for any $s\in S$ and $x\in X$. Each word $W=x_1 x_2\cdots x_n$ in $X^*$ represents a path in $\Gamma$ which has the following vertices:
\[
1, \overline{x_1}, \overline{x_1 x_2}, \ldots, \overline{x_1 x_2\cdots x_n}
\]
There are several ways to define a metric on $\Gamma$. One option is to consider the distance from $s_1$ to $s_2$ as the length of shortest positive path connecting between them. Another option is to consider the path metric on $\Gamma$ viewed as non-directed graph. These two options have their advantages and limitations. Here we consider another option which is only available when the semigroup is embeddable. For that end, consider a semigroup $S$ finitely generated by $X$ with a semigroup presentation
\[
\Sgp\Pres{X | R_1=S_1,\ldots, R_n=S_n}
\]
Then, the co-presented group of $S$ is the group $G$ with presentation
\[
\Grp\Pres{X | R_1=S_1,\ldots, R_n=S_n}
\]
If $S$ is embeddable then it is embeddable in the co-presented group as the sub-semigroup of positive words (positive words are words in $X^*$ and negative words are words in $(X^{-1})^*$). Consequently, $\Gamma(S,X)$ is embedded in $\Gamma(G,\pmX)$ and we can define a metric on $\Gamma(S,X)$ which is induced from the word metric on $\Gamma(G,\pmX)$. We term this metric as the \emph{induced metric} on $\Gamma(S,X)$.

We conclude the preliminary section with few important (but easy) consequences of the definition of a piece. Suppose $W$ is a subword of a defining word such that $\lp{W}\geq2$. If follows that $W$ fixes two \emph{unique} elements $U_1$ and $U_2$ such that $U_1 W U_2 \in \R$ (it is possible that $U_1 = \eps$ or $U_2 = \eps$). The reason is that non-uniqueness would imply, by the definition of pieces, that $W$ is a piece and consequently $\lp{W}=1$. So we have:

\begin{observation} \label{obs:twoPiecUniq}
Let $W$ be subword of a defining word such that $\lp{W}\geq2$. Then, there is a \emph{unique} $R \in \R$ such that $W$ is a subword of $R$.
\end{observation}

Here is another observation that follows from the definition of a piece and condition $K_3^2$.

\begin{observation} \label{obs:sharedPieceLP1}
Assume that condition $K_3^2$ holds and suppose $W=W' W''$ is a subword of a defining word where $W' \neq \eps$ and $W'' \neq \eps$. If $W''$ is a prefix of some defining word then $\lp{W''} = 1$. Similarly, if $W'$ is a suffix of some defining word then $\lp{W'} = 1$.
\end{observation}
\begin{proof}
We prove the first case. Take $V_1,V_2$ and $U$ such that $V_1 W V_2$ is a defining word and $W'' U$ is a defining word. Clearly $V_2 \neq U$ (because, otherwise we would get that $W'' U$, a defining word, is a subword of $V_1 W V_2$, another defining word, and thus $\lp{W'' U}=1$ which contradicts the $K_3^2$ condition). Hence, by the definition of a piece we get that $W''$ is a piece and thus $\lp{W''}=1$.
\end{proof}

\section{Automaticity in Embeddable semigroups} \label{sec:autoInSemigrp}

Automatic groups have a so-called geometric characterization through the idea of fellow-travelling paths (see \cite[Ch. 2]{EPS92} and Definition \ref{def:KFT} below). For semigroups and monoids such a simple geometric characterization does not apply. However, Hoffmann and Thomas \cite{HT06}, and Silva and Steinberg \cite{SS04} independently gave similar---though less elegant---geometric characterizations for semigroups and monoids; in their work additional conditions are needed on top of fellow-travelling. For embeddable semigroups a group-like geometric characterization can be given (Theorem \ref{thm:autoInEmbdlSG}). First, here is the definition of fellow-travellers in semigroups:

\begin{definition}[Fellow-Travellers \cite{EPS92}]\label{def:KFT}
Let $S$ be a semigroup finitely generated by $X$ and let $d(\cdot,\cdot)$ be some metric on the Cayley graph $\Gamma$. For a word $W\in X^*$ we denote by $W(n)$ the prefix of length $n$ of $W$ (which is $W$ if $n\geq|W|$, the length of $W$). Two words $W$ and $U$ in $X^*$ are called \emph{$k$-fellow-travellers} (relative to $d$) if for any $n\in\mathbb{N}$:
\[
d(\overline{W(n)},\overline{U(n)}) \leq k
\]
A set of words $L\subseteq X^*$ has the fellow-traveller property if there is some constant $k$ such that for each $W$ and $U$ in $L$ such that $d(\overline{W},\overline{U})\leq1$ we have that $W$ and $U$ are $k$-fellow-travellers. In the sequel we will write $d(W,U)$ instead of $d(\overline{W},\overline{U})$.
\end{definition}

Here is a useful feature of the fellow-travelling property.

\begin{lemma} \label{lem:twoKFTs}
Suppose $W$ and $U$ are $k$-fellow-travellers and also $U$ and $V$ are $\ell$-fellow-travellers (with respect to some metric $d$) then $W$ and $V$ are $(k+\ell)$-fellow-travellers.
\end{lemma}
\begin{proof}
This follows from:
\[
d(W(i),V(i)) \leq d(W(i),U(i)) + d(U(i),V(i)) \leq k + \ell
\]
\end{proof}

Next we give the characterization of automaticity in embeddable semigroups. The following theorem seems to be folklore; we give its proof for completeness.

\begin{theorem} \label{thm:autoInEmbdlSG}
Let $S$ be an embeddable semigroup, finitely generated by $X$. Then, $S$ is automatic if and only if there is a regular language $L\subseteq X^*$ such that $\pi_{X,S}(L)=S$ and $L$ has the fellow-traveller property under the induced metric on the Cayley graph $\Gamma$.
\end{theorem}

The ``only if'' part of the proof of Theorem \ref{thm:autoInEmbdlSG} follows immediately from Lemma 3.12 in \cite{CRRT01}. We prove the ``if'' part. For embeddable semigroups we have the following lemma:
\begin{lemma}
Suppose $S$ is an embeddable semigroup finitely generated by $X$ and consider the induced metric on the Cayley graph $\Gamma$. Let $k$ be a natural number and let $x\in X\cup\Set{\eps}$. The following language, denoted by $FT_{x}^k(S,X)$, is regular:
\[
FT_{x}^k(S,X) = \Set{(W,U)\delta_X | \begin{array}{l}
  \textrm{$W$ and $U$ are $k$-fellow-travellers and } \\
  Wx\eqb{S} U \\
\end{array}
}
\]
\end{lemma}
\begin{proof}
Denote by $G$ the co-presented group of $S$ generated as a semigroup by $X^{\pm1}$. $J_x$ denote the following language:
\[
J_{x} = \Set{(W,U)\delta_{X^{\pm1}} | \begin{array}{l}
  \textrm{$W$ and $U$ are $k$-fellow-travellers and } \\
  Wx\eqb{G} U \\
\end{array}
}
\]
It is well known that $J_x$ is regular (see the proof of Theorem 2.3.4 in \cite{EPS92}). Hence, it follows that $FT_{x}^k(S,X)$ is regular since $FT_{x}^k(S,X) = J_x \cap (X^*\times X^*)\delta_X$.
\end{proof}

\begin{proof}[Proof of Theorem \ref{thm:autoInEmbdlSG} (if part)]
Let $S$ be an embeddable semigroup, finitely generated by $X$. Suppose $L\subseteq X^*$ is a regular language which is onto $S$ through the natural map and which has the fellow-traveller property for some constant $k$. We show that $S$ is automatic by showing that for all $x\in X\cup\Set{\eps}$ the set $L_{x}=\Set{(W,U)\delta_X|W,U\in L;\; Wx=_S U}$ is regular. First, notice that $L_{x}\subseteq FT_{x}^k(S,X)$ since $L$ has the fellow-traveller property. Next, since intersection preserves regularity \cite[Thm.\ 4.8]{HU} the set
\[
(L\times L)\delta_X \cap FT_{x}^k(S,X)
\]
is regular. Finally, the elements in $L_x$ are elements of $FT_{x}^k(S,X)$ having the form $(W,U)\delta_X$ with $W$ and $U$ in $L$ and thus there is an equality $L_{x} \cap FT_{x}^k(S,X) = (L\times L)\delta_X \cap FT_{x}^k(S,X)$. This implies that
\[
L_{x} = (L\times L)\delta_X \cap FT_{x}^k(S,X)
\]
and consequently $L_{x}$ is regular.
\end{proof}

Next, we show how to generate an automatic structure for embeddable semigroups through regular partial orders (an order ``$\prec$'' on $X^*$ is regular if the set $\Set{(W,U)\delta_X|W\prec U}$ is regular). This technique is called `falsification by fellow travellers' and is based on a work by Davis and Shapiro (see also \cite{Pei96,Wei07}).

\begin{theorem} \label{thm:falseification}
Let $S$ be an embeddable semigroup finitely generated by $X$. Suppose ``$\prec$'' is a regular partial order on $X^*$. Denote by $M_\prec$ the following set:
\[
M_\prec = \Set{W\in X^* | \textrm{for all $U\in X^*$ if $W=_S U$ then $W\prec U$} }
\]
We assume that $\pi_{X,S}(M_\prec)=S$. Suppose there is a constant $k$ such that the following properties of ``$\prec$'' holds:
\begin{enumerate}
\item[(R)] If $W\not\in M_\prec$ then there is $U\in X^*$ such that $W=_S U$,  $U \prec W$, and $W,U$ are $k$-fellow-travellers.

\item[(FT)] If $W$ and $U$ in $M_\prec$ and $Wa=_S U$ for some $a\in X\cup\Set{\eps}$ then $W$ and $U$ are $k$-fellow-travellers.
\end{enumerate}
Then, $S$ is an automatic semigroup.
\end{theorem}
\begin{proof}
By assumption the set $M_\prec$ is onto $S$ through the natural map. By Property (FT). the set $M_\prec$ has the fellow-traveller property. Hence, to establish automaticity it is enough by Theorem \ref{thm:autoInEmbdlSG} to show that $M_\prec$ is regular. We denote the set $FT_{\eps}^k$ by $K$ and the set $\Set{(W,U)\delta_X | W\prec U}$ by $P$ (recall that $P$ is regular since we assumed that ``$\prec$'' is regular). Since intersection preserve regularity, the following
set is regular:
\[
K\cap P = \Set{(W,U)\delta_X | W\prec U, W=_S U, \textrm{ and }
W,U \textrm{ are $k$-fellow-travellers}}
\]
Projection also preserves regularity \cite[Prop.\ 2.2(vii)]{CRRT01} and therefore the following set is regular:
\[
C = \Set{U | \exists W : (W,U)\in K\cap P}
\]
By Property (R) an element $W$ is in $C$ if and only if it is not in $M_\prec$. Hence, by Property (R) the set $C$ is exactly the complement of $M_\prec$. Consequently, $M_\prec$ is regular since taking complement preserves regularity \cite[Thm.\ 4.5]{HU}.
\end{proof}

We will call an element of $M_\prec$ an ``$\prec$''-minimal element (reads as ``order minimal''). The theorem above shows that the set of ``$\prec$''-minimal elements is an automatic structure (assuming, of course, that the conditions of the theorem hold).

\section{Example of non-geodesic structure} \label{sec:example}

The $K_3^2$ semigroups considered in the main theorem are embeddable semigroups. We give in this section an example of a semigroup $S$ which is automatic by the main theorem but for which for a given set of generators the set of geodesics is not an automatic structure. This is in sharp contrast to the situation in $C(7)$ groups. To recall the definition, a word $W$ is geodesic if for every $U$ such that $W=U$ in $S$ we have that $|W|\leq|U|$.

The semigroup we consider is the semigroup with the following presentation:
\[
\Pres{a,b,c | abcc=cba}
\]
Here, $\R=\Set{abcc,cba}$ and $X=\Set{a,b,c}$. There are only three pieces: $a$, $b$, and $c$. Thus, the $K_3^2$ conditions holds by simple inspection and so $S$ is an embeddable semigroup (embeddable in this case in a hyperbolic group since the co-presented group is a $C(7)$ group). We give two geodesics, $V_n$ and $U_n$ of lengths $3n$ and $2n+1$, respectively, such that $V_n c = U_n$ in $S$ (i.e. $d(U_n,V_n)=1$ in the Cayley graph). Hence, if $k$ is fixed and $n$ is large enough then $V_n$ and $U_n$ are not $k$-fellow-travellers (due to the big difference in their lengths). The definitions of $V_n$ and $U_n$ follows: let $n$ be some natural number and let $V_n=(abc)^n$ and $U_n=c(ba)^n$. See Figure \ref{fig:geoInC7Sgp} for an illustration of part of the Cayley graph of $S$ containing $V_n$ and $U_n$. 

\begin{figure}[ht]
\centering
\includegraphics[totalheight=0.13\textheight]{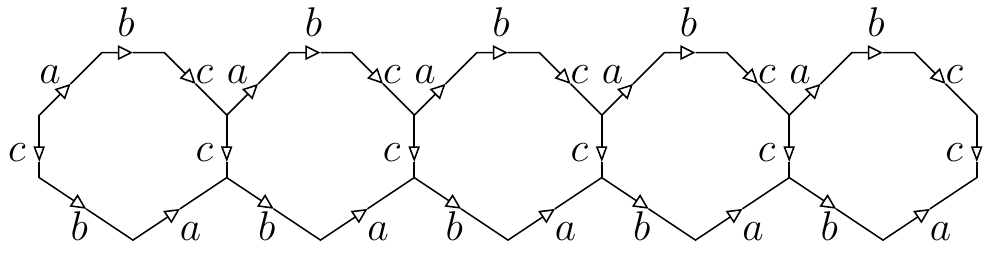}
\caption{Two non-fellow-travelleing geodesics} \label{fig:geoInC7Sgp}
\end{figure}

Denote by $s_n$ and $t_n$ the elements in $S$ presented by $U_n$ and $V_n$, respectively. Now, since the two sides of the relation are not subwords of $V_n$ we have that no other element in $X^*$ presents $t_n$. Thus, $V_n$ is a geodesic. For $U_n$ we do have a subword that is one side of the relation. However, by applying the relation one can only increase the length. Thus, $U_n$ is also a geodesic. Using the relation $abcc=cba$ we get that $(abc)^nc=c(ba)^n$ so consequently $V_n c = U_n$, as claimed. 

Another interesting observation regarding the above example is the following: by the above discussion there are no other geodesics presenting $s_n$ and $t_n$ in $X^*$. Therefore, any automatic structure for $S$ under the generating set $X$ cannot contain just geodesics, since it will then include $U_n$ and $V_n$, which is impossible.

\section{van Kampen Diagrams} \label{sec:vanKampen}

We use the theory of van Kampen diagrams, both for semigroups and groups. See \cite[Chapter V, p. 235]{LS77} for a standard introduction of van Kampen diagrams of groups and see \cite{Rem71} or \cite[p. 73-79]{Hig92} for the van Kampen diagram theory for semigroups. Here we give a unified treatment for both cases. A \emph{diagram} is a finite planar connected and simply connected $2$-complex. We name the $0$-cells, $1$-cells, and $2$-cells by \emph{vertices}, \emph{edges}, and \emph{regions}, respectively. Vertices of valence one or two are allowed. Each edge has an orientation, i.e., a specific choice of initial and terminal vertices. Given an edge $e$ we denote by $i(e)$ the initial vertex of $e$ and by $t(e)$ the terminal vertex of $e$. If $e$ is an oriented edge then $e^{-1}$ will denote the same edge but with the reverse orientation. A \emph{path} is a series of (oriented) edges $e_1,e_2,\ldots,e_n$ such that $t(e_j) = i(e_{j+1})$ for $1\leq j < n$. The length of a path $\rho$ (i.e., the number of edges along $\rho$) is denoted by $|\rho|$. If $\rho$ is the path $e_1 \cdots e_n$ then we denote by $\rho^{-1}$ the path $e_n^{-1} \cdots e_1^{-1}$. If $\rho$ is a path that decomposes as $\rho=\rho_1\rho_2$ then $\rho_1$ is a \emph{prefix} of $\rho$ and $\rho_2$ is a \emph{suffix} of $\rho$.

Given a finite \emph{group} presentation $\Grp\Pres{X |\R}$, a \emph{group diagram} over this presentation is a diagram where its edges are labelled by elements of $\pmX$ and the boundary of every region is labelled by elements of the symmetric closure of $\R$. We also require that if an edge $e$ is labelled by $x$ then $e^{-1}$ is labelled by $x^{-1}$. In the context of group diagram we say that an edge $e$ is \emph{positive} (resp., \emph{negative}) if its label is in $X$ (resp., in $X^{-1}$). In the same manner, a path is positive (resp., negative) if it consists of positive (resp., negative) edges. A \emph{boundary label} is the label of some path $\rho$ that coincides with the boundary of the diagram. Next we give the definition of semigroup diagrams; these require some additional assumptions. Suppose we are given a \emph{semigroup} presentation $\Sgp\Pres{X|L_1=R_1,\ldots,L_n=R_n}$. A \emph{semigroup diagram} $M$ over the given presentation is a group diagram over the co-presented group such that three conditions hold: (1) there is a boundary label $WU^{-1}$ where $W$ and $U$ are positive; (2) any inner vertex is an initial vertex of some positive edge (i.e., there are no inner sink vertices); (3) any inner vertex is a terminal vertex of some positive edge (i.e., there are no inner source vertices). van Kampen theorem state that equality $W=U$ holds in a group (semigroup) if and only if there is a van Kampen group (semigroup) diagram with boundary label $WU^{-1}$ (such diagrams are called equality diagrams). If we don't explicitly indicate for a given diagram whether it is a group diagram or a semigroup diagram then it may be either one of the two options.

In the sequel, if a word $W$ labels a path on the boundary of $M$ then $W$ would denote both the path and the word; the context would make the distinction clear. The term \emph{neighbors}, when referred to two regions, means that the intersection of the regions' boundaries contain an edge; specifically, if the intersection contains only vertices, or is empty, then the two regions are not neighbors. \emph{Boundary regions} are regions with outer boundary, i.e., the intersection of their boundary and the diagram's boundary contains at least one edge. If $D$ is a boundary region in $M$ then the \emph{outer-boundary} of $D$ is $\partial D \cap \partial M$ and the \emph{inner boundary} of $D$ is the rest of the boundary (i.e., the complement of the outer boundary). Regions which are not boundary regions will be called \emph{inner regions}. In a similar manner, a \emph{boundary edge} is an edge in the boundary of the diagram and an \emph{inner edge} is an edge not on the boundary. A \emph{minimal diagram} is a diagram with minimal number of regions among the diagram with the same boundary label.

Suppose $D$ and $E$ are neighboring regions in $M$ and let $\delta$ be a connected component of $\partial D \cap \partial E$. It is a well known fact that if $M$ is a minimal group diagram then the label of $\delta$ is a \emph{piece} (see the introduction for the definition). This may not be the case for general semigroup diagrams. However, as we shall shortly see, in the cases we consider the label of $\delta$ is always a peice.

\begin{definition}[Strong s-condition \cite{Kas92}]
Let $\P$ be a semigroup presentation of a semigroup $S$. We say that $\P$ has the \emph{strong s-condition} if the following hold. Suppose $W$ and $U$ are positive words and $W=U$ is an equality in the co-presented group of $S$. Suppose further that $M$ is a (group) van Kampen diagram over the \emph{co-presented group} of $S$ with $WU^{-1}$ as boundary cycle. Then, there is a boundary region $D$ in $M$ with boundary cycle $\rho \delta^{-1}$ such that the labels of $\rho$ and $\delta$ are positive and $\rho$ is the outer boundary of $D$ and a subword of $W$ or $U$. See figure \ref{fig:grpDiagToSgpDiag}.
\end{definition}

\begin{figure}[ht]
\centering
\includegraphics[totalheight=0.18\textheight]{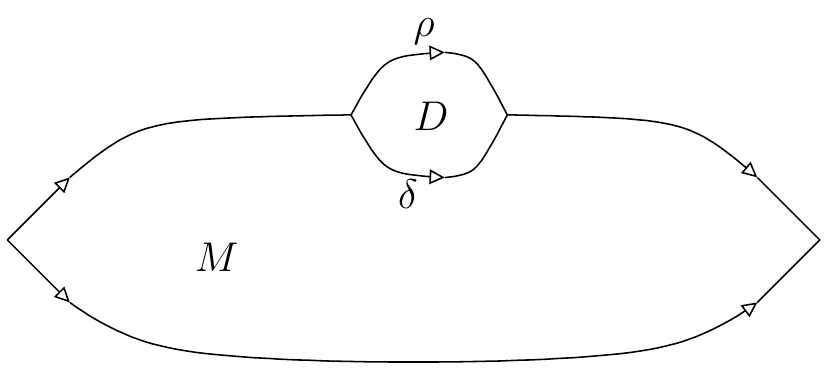}
\caption{Illustration of strong s-condition} \label{fig:grpDiagToSgpDiag}
\end{figure}

The main result of \cite{Gub94} is that $K_3^2$ semigroups have the strong s-condition (and, hence, are embeddable \cite{Kas92}). Suppose we are given a $K_3^2$ semigroup $S$ and we consider a group diagram $M$ with boundary cycle $WU^{-1}$ over the co-presented group of $S$ ($W$ and $U$ are positive words). By the strong s-condition we have some region $D$ that we can remove from $M$ such that the resulting diagram $\tilde{M}$ has a boundary label $\tilde{W}\tilde{U}^{-1}$ where $\tilde{W}$ and $\tilde{U}$ are positive (and, clearly, $\tilde{M}$ has less regions than $M$). This is the essence of the proof of the following lemma:

\begin{lemma}\label{lem:goodVanKampenDiagK23}
Let $S$ be semigroup with a $K_3^2$ presentation $\P$ and let $M$ be a minimal semigroup van Kampen diagram over $\P$. Suppose $D$ and $E$ are neighboring regions in $M$ and that $\delta$ is a connected component of $\partial D \cap \partial E$. Then, $\delta$ is labelled by a piece.
\end{lemma}
\begin{proof}[Proof (sketch)]
Denote by $G$ the co-presented group of $S$. It is enough to show that $M$ is a minimal group diagram over $G$ (because, in minimal group diagrams every edge is labelled by a piece). Denote by $|M|$ the number of regions in $M$. Let $N$ be a minimal group diagram over $G$ with the same boundary label as $M$. Clearly, $M$ is a group diagram over $G$ so it remains to show that it is minimal, or in other words to show that $|M|=|N|$. It is also clear that $|N|\leq |M|$. Thus, we need to show that $|M|\leq |N|$. We will do that by showing that $N$ is a semigroup diagram. We prove that $N$, a minimal \emph{group} diagram, is a semigroup diagram by induction on $|N|$. If $|N|=0$ (i.e., there are no regions in $N$) then clearly there are no inner source or sink vertices in $N$ so $N$ is a semigroup diagram. Suppose that the assertion is true when $|N|<n$ and we have that $|N|=n$. We use the strong s-condition and we denote the region it guarantees by $D$. We remove the region $D$ from $N$ and denote the new diagram by $N'$. Clearly, $|N'|<n$ so by induction hypothesis we get that $N'$ is a semigroup diagram. Finally, by attaching $D$ back to $N'$ (which restores the diagram $N$) we see that $N$ is also a semigroup diagram.
\end{proof}

A \emph{$(\mu,\sigma)$-thin} diagram is a diagram $M$ with boundary cycle $\mu\sigma^{-1}$ where every region $D$ has at most two neighbors and $\partial D$ has non-empty intersection with $\mu$ and $\sigma$. See an illustration of such diagram in Figure \ref{fig:thinEqualityDiagram}. The notion of thin diagrams (also known as one layered diagrams) appeared in \cite{Pei96,Wei07} and in several other earlier works.

\begin{figure}[ht]
\centering
\includegraphics[totalheight=0.18\textheight]{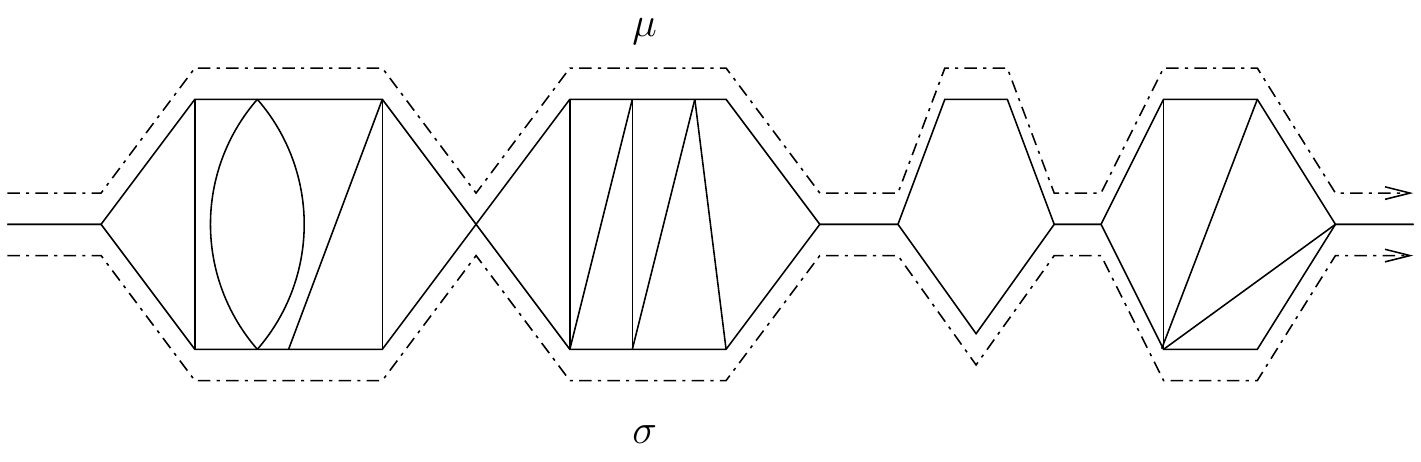}
\caption{Thin equality diagram} \label{fig:thinEqualityDiagram}
\end{figure}

Let $M$ be a diagram and $D$ a region in $M$. Suppose $\omega$ is a subpath of $\partial D$ with label $W$. We denote by $N_\omega^D$ the number of neighbors of $D$ along $\omega$ counted with multiplicity; see Figure \ref{fig:pcsBndNeis} for an illustration of neighbors along a path. 
\begin{figure}[ht]
\centering
\includegraphics[totalheight=0.18\textheight]{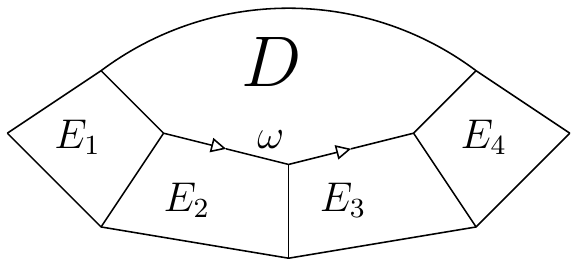}
\caption{The neighbors of $D$ along a subpath $\omega$ of $\partial D$} \label{fig:pcsBndNeis}
\end{figure}
Suppose next that $M$ is a minimal \emph{semigroup} van Kampen diagram over a presentation for which the conditions of the main theorem hold. By lemma \ref{lem:goodVanKampenDiagK23} we have that the neighbors of $D$ along $\omega$ induce a decomposition of $W$ into pieces and thus we get that $N_\omega^D \geq \lp{W}$. An immediate implication of this is that $M$ is a $C(7)$ diagram (i.e., a diagram where every inner region has at least seven neighbors) and we can use the tools of small cancellation theory for these diagrams. The main diagrammatic result of this section is the following:

\begin{proposition} \label{prop:condForThin}
Assume a semigroup $S$ is a semigroup with presentation $\Sgp\Pres{X|L_1=R_1,\ldots,L_n=R_n}$ for which the conditions of the main theorem hold. Let $W$ and $U$ be two positive words and let $a\in X\cup\Set{\eps}$. If $M$ is a minimal semigroup diagram over the presentation with boundary cycle $WaU^{-1}$ then there are two options:
\begin{enumerate}
 \item $M$ is a $(Wa,U)$-thin diagram.
 \item There is a boundary region $D$ with $\partial D=\rho\delta^{-1}$ such that:
 \begin{enumerate}
  \item $\rho$ is the outer boundary of $D$ and is a subpath of $W$ or $U$.
  \item $\delta$ is the inner boundary and $N_\delta^D = 3$.
  \item $\rho$ and $\delta$ are labelled by positive words.
 \end{enumerate}
\end{enumerate}
\end{proposition}

Before we can give the proof of Proposition \ref{prop:condForThin} we need two diagrammatic results. The first is a lemma from \cite{Hig92} (see also Lemma 4.8(a) of \cite{Rem71}).

\begin{lemma}[Lemma 5.2.10(i) of \cite{Hig92}]\label{lem:RgnsInC3Diag}
Let $S$ be a semigroup with presentation $\P$ such that any defining word has piece-length at least three. Suppose $W=U$ is an equality in $S$ and $M$ is a minimal van Kampen diagram with boundary label $WU^{-1}$. Let $D$ be a region in $M$ with boundary cycle $\rho\delta^{-1}$ where $\rho$ and $\delta$ are positively labelled. Then, $N_\rho^D \leq 3$ and $N_\delta^D \leq 3$.
\end{lemma}

The consequence of Lemma \ref{lem:RgnsInC3Diag} is that minimal van Kampen diagrams over presentations having the conditions of the main theorem have no inner regions. Namely, in these diagrams every region has a boundary. Next we state a lemma which gives information on the structure of a $C(7)$ diagram. This lemma can be deduced from the proof of Greendlinger Lemma \cite[Thm. 4.5]{LS77}. A direct proof of (a generalization of) the lemma can be found in \cite[Thm. 13]{Wei07}. Note however that the theorem in \cite{Wei07} needs to be translated to the terminology used here; we suppress the technical details.

\begin{lemma}[Greendlinger's lemma] \label{lem:Greendlinger}
Let $M$ be a $C(7)$ diagram with boundary cycle $\mu\xi\sigma^{-1}$ where $\xi$ is labelled by  a generator or is empty. Then, one of the following holds:
\begin{enumerate}
	\item $M$ is a $(\mu\xi,\sigma)$-thin diagram.
 \item There is a boundary region $D$ with $\partial D=\rho\delta^{-1}$ such that:
 \begin{enumerate}
  \item $\rho$ is the outer boundary of $D$ and is a subpath of $\mu$ or $\sigma$.
  \item If $\xi$ is labelled by a generator then $\rho$ does not contain $\xi$.
  \item $D$ has at most three neighbors in $M$.
 \end{enumerate}
\end{enumerate}
\end{lemma}

Equipped with these two results we can now give the proof of Proposition \ref{prop:condForThin}.

\begin{proof}[Proof of Proposition \ref{prop:condForThin}]
By Lemma \ref{lem:goodVanKampenDiagK23} every inner edge in $M$ is labelled by a piece and thus $M$ is a $C(7)$ diagram. Assume that $M$ is not a $(Wa,U)$-thin diagram. By Lemma \ref{lem:Greendlinger} there is a boundary region $D$ in $M$ such that $D$ has at most
three neighbors and the outer boundary of $D$ is contained in $W$ or in $V$. Suppose the boundary of $D$ is $\rho\delta^{-1}$ where both $\rho$ and $\delta$ are positively labelled. Since the outer boundary of $D$ is contained in $W$ or in $V$ we get that one of the parts, $\rho$ or $\delta$, of the boundary of $D$ is completely contained in the inner boundary of $D$. Assume w.l.o.g. that $\delta$ is contained in the inner boundary of $D$. Let $V$ be the label of $\delta$ (this is a defining word). Using Lemma \ref{lem:RgnsInC3Diag} and the $K_3^2$ condition we have that $3 \leq \lp{V}\leq N_\delta^D \leq 3$. Consequently, $\delta$ is exactly the inner boundary of $D$ (because it has exactly three neighbors so all the neighbors of $D$ intersect with $\delta$ and not with $\rho$). As a result, $\rho$ does not contain inner edges. This proves the proposition.
\end{proof}

We finish the section on diagrams with the next lemma. The lemma characterizes the structure of thin equality diagrams over presentations which satisfy the conditions of the main theorem and another technical condition (one which later we can assume).

\begin{lemma} \label{lem:essencialVert}
Let $\P$ be a semigroup presentation which satisfy conditions $K_3^2$ and \emph{(\dag)} of the main theorem and let $M$ be a $(\mu\xi,\sigma)$-thin diagram over this presentation. We will assume that $\xi$ is empty or the label of $\xi$ is a generator. We will also assume the following technical condition:
\begin{enumerate}
	\item[\emph{(\ddag)}] Suppose $D$ is a region in $M$ with boundary path $\delta\rho^{-1}$ such that $\delta$ and $\rho$ are (positively) labelled by $V_\delta$ and $V_\rho$ and $\lp{V_\delta}>\lp{V_\rho}$. Then, $\delta$ is not a subpath of $\mu$ and is not a subpath of $\sigma$.
\end{enumerate}
Then:
\begin{enumerate}
 \item \label{lem:essencialVert:vt} If $\nu$ is a vertex of $\mu$ of valence at least three that is not a vertex of $\sigma$ then $\nu$ is of valence exactly three. Specifically, if $D$ is a region of $M$ then $\partial D \cap \mu$ and $\partial D \cap \sigma$ both contain an edge.

 \item \label{lem:essencialVert:lt} If $D$ is a region in $M$ then the label of $\partial D \cap \mu$ has piece-length at least two.

 \item \label{lem:essencialVert:br} If $D$ is a region in $M$ that has at most one neighbor, and its boundary does not contain $\xi$. Then, the piece-length of the label of $\partial D \cap \mu$ is at least three.

 \item \label{lem:essencialVert:gr} Suppose that $D_1$ and $D_2$ are two neighboring regions and let $V_1$ and $V_2$ be the labels of $\partial D_1 \cap \mu$ and $\partial D_2 \cap \mu$, respectively. The word $W=V_1 V_2$ has the property that if $U$ is a prefix of $W$ which is a subword of defining word then $|U|\leq |V_1|$ (we will later denote such decomposition of $W$ as left-greedy decomposition).

\end{enumerate}
\end{lemma}
\begin{proof}
We prove the different parts one by one:
\begin{enumerate}

\item See figure \ref{fig:vertValanceThree}. Assume by contradiction that $\nu$ is a vertex of $\mu$ of valence bigger than three which is not a vertex of $\sigma$. In this case there is a region $D$ with two inner edges that are adjacent to $\nu$. Thus, if $\partial D = \rho \delta^{-1}$ such that both $\rho$ and $\delta$ are labelled by positive words then one of them would have piece-length at most two (because the diagram is thin so $D$ has at most two neighbors and the inner parts of $\partial D$ are labelled by pieces). This contradicts the $K_3^2$ condition.

\begin{figure}[ht]
\centering
\includegraphics[totalheight=0.18\textheight]{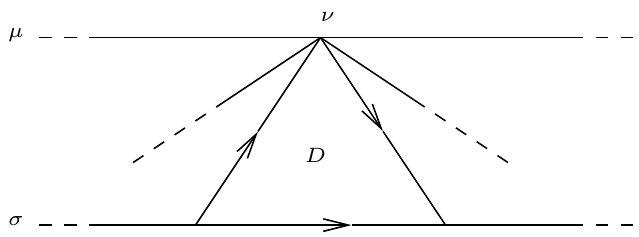}
\caption{An impossible situation where a vertex has valence
greater than three} \label{fig:vertValanceThree}
\end{figure}

\item See figure \ref{fig:maxTwoVert}. Let $D$ be a region in $M$. We denote by $\omega_u$, $\omega_d$, $\omega_\ell$, and $\omega_r$ the four sides of $D$ such that $\omega_u = \partial D \cap \mu$, $\omega_d = \partial D \cap \sigma$ and the inner boundary of $D$ consists of $\omega_\ell$ and $\omega_r$ (they may be empty and $\omega_r$ may equal $\xi$). Denote by $V_u$, $V_d$, $V_\ell$, and by $V_r$ the labels of $\omega_u$, $\omega_d$, $\omega_\ell$, and $\omega_r$, respectively. Clearly, $\lp{V_\ell}\leq1$ and $\lp{V_r}\leq1$ (since they are pieces or empty). We need to show that $\lp{V_u}\geq2$. Assume otherwise by contradiction, namely, that $\lp{V_u}\leq1$. Depending on $V_\ell$ and $V_r$ being positive or negative words, we have that one of the following is a defining relation in $S$:
\[
V_u V_r = V_\ell V_d, \quad V_\ell V_u = V_d V_r, \quad V_u = V_\ell V_d V_r, \quad V_\ell V_u V_r = V_d,
\]
The first three cannot be a defining relations since the left side decomposes into less than three pieces and thus violate the $K_3^2$ condition. Hence, $V_\ell V_u V_r = V_d$. But, $\lp{V_\ell V_u V_r} \leq 3$ and so by condition $(\dag)$ of the main theorem we have that $\lp{V_d}\geq4$. This contradicts assumption (\ddag).

\begin{figure}[ht]
\centering
\includegraphics[totalheight=0.18\textheight]{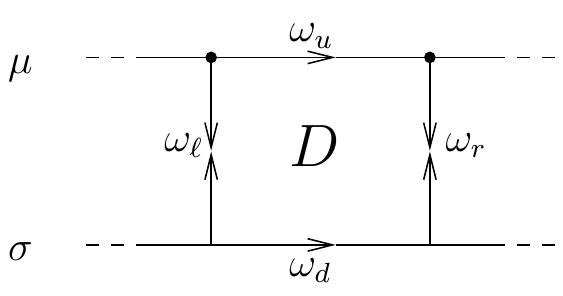}
\caption{The pieces length of $\partial D \cap \mu$} \label{fig:maxTwoVert}
\end{figure}

\item See Figure \ref{fig:oneNeiReg}. Assume the boundary of $D$ does not contain $\xi$ and suppose $D$ is a region that has at most one neighbor. The boundary of $D$ decomposes into three parts: $\omega_\mu = \partial D \cap \mu$, $\omega_\sigma = \partial D \cap \sigma$, and, possibly empty, inner part $\omega_{in}$ (which, if not empty, is labelled by a piece). We need to show that the piece-length of $\omega_\mu$ is at least three. Assume otherwise by contradiction. If follows from the $K_3^2$ condition that either $\omega_\mu \omega_{in}$ or $\omega_{in} \omega_\mu$ is positively labelled. Assume w.l.o.g. that $\omega_\mu \omega_{in}$ is positively labelled. By assumption $\omega_\mu \omega_{in}$ has a piece-length at most three. Thus, by condition (\dag) of the main theorem we get that $\omega_\sigma$ is labelled by a positive label of piece-length at least four. This contradicts assumption (\ddag).

\begin{figure}[ht]
\centering
\includegraphics[totalheight=0.18\textheight]{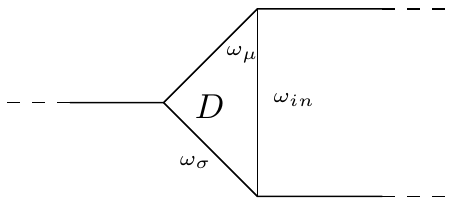}
\caption{A region with at most single neighbor} \label{fig:oneNeiReg}
\end{figure}

\item See Figure \ref{fig:leftGreedy}. Let $\delta_1 = \partial D_1 \cap \mu$, $\delta_2 = \partial D_2 \cap \mu$, and $\omega = \partial D_1 \cap \partial D_2$. Assume that $i(\omega)$, the first vertex of $\omega$, is a vertex of $\mu$. Suppose by contradiction that there is a prefix $U$ of $W$ which is a subword of a defining word and $|U| > |V_1|$. Thus, we can decompose $\delta_2$ into $\delta_2 = \delta_2'\delta_2''$ such that the label of $\delta_1\delta_2'$ is $U$. Since $\lp{V_1}\geq2$ we have Observation \ref{obs:twoPiecUniq} that there is a unique defining word $R$ such that $V_1$ is its subword. Thus, $U$ is also a subword of $R$. Consequently, $U$ is a prefix of the label of $\delta_1\omega$. We get that $\delta_2$ and $\omega$ are positively labelled and start with the same generator. This is a contradiction to the $K_3^2$ condition.

\begin{figure}[ht]
\centering
\includegraphics[totalheight=0.18\textheight]{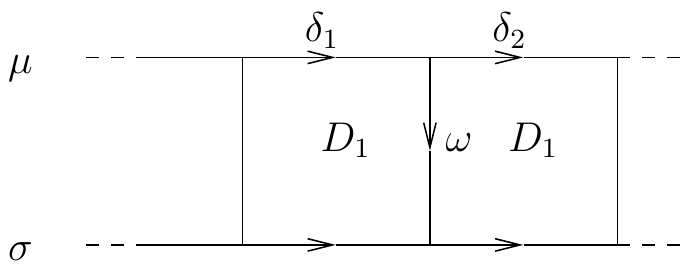}
\caption{Left greedy decomposition of the labels} \label{fig:leftGreedy}
\end{figure}

\end{enumerate}

\end{proof}

\section{Proof of Main Theorem} \label{sec:proofMainThm}

In this section we prove Theorem \ref{thm:mainThm}. For the rest of this section fix a presentation $\P = \Pres{X| L_1=R_1, \ldots, L_n=R_n}$ for a semigroup $S$ for which the conditions of the main theorem hold (conditions $K_3^2$ and (\dag)). As suggested by Theorem \ref{thm:falseification}, we will prove automaticity by producing a regular order on some generating set. We will assume that any generator $x \in X$ appears in one of the defining relations. If that doesn't happen then we can split $S$ as $S=S'*F$ where $S'$ has this property, the conditions of the main theorem hold for $S'$, and $F$ is a finitely generated free semigroup. Since a free product of automatic semigroups is automatic \cite[Thm. 6.1]{CRRT01} it is enough to prove the theorem for $S'$.

We start by defining the generating set we will be working with. Let $\Sb$ be the (finite) set of subwords of the elements in $\R=\R(\P)=\Set{L_1,\ldots,L_n} \cup \Set{R_1,\ldots,R_n}$ and let $\Phi=\Set{\varphi_W | W\in\Sb}$. In other words, $\Phi$ is a set of symbols which corresponds to the subwords of elements in $\R$. Let $\pi:\Phi^* \to S$ be the natural map for which $\pi(\varphi_W)$ is the element that $W$ presents in $S$. By our assumption above we have that $X\subseteq\Sb$ and thus the set $\pi(\Phi)$ is a generating set for $S$. Our automatic structure will be a subset of $\Phi^*$ and it will be constructed by defining an order on the words in $\Phi^*$. We write $d_\Phi(\cdot,\cdot)$ to denote the induced metric on the Cayley graph of $S$ under the new generating set (see the end of Section \ref{sec:preliminaries}). The symbols $A$, $B$ and $C$ will denote elements of $\Phi^*$ and the symbols $U$, $V$, $W$ will denote elements of $X^*$. If $A \in \Phi^*$ then there are elements $W_1,\ldots,W_n$ of $\Sb$ such that $A=\varphi_{W_1}\cdots\varphi_{W_n}$. In this case we will use the notation $\eta(A)$ to denote the word $W_1\cdots W_n$ in $X^*$. Thus, $\pi(A)=\overline{\eta(A)}$.

Recall from property (c) of the $K_3^2$ condition that if $W \in \R$ is a defining word then there is a \emph{unique} defining word $U \in \R$ such that $W=U$ or $U=W$ is a defining relation. In this case we say that $U$ is the \emph{complement} of $W$. We denote the complement of $W$ by $\Comp{W}$. Note that $\Comp{\Comp{W}}=W$.

We next define for each word in $\Phi^*$ an auxiliary vector. These vectors will be used to define an order on $\Phi^*$.

\begin{definition}[Auxiliary vector for $\Phi^*$] \label{def:auxVect}
Let $A=\varphi_{W_1}\cdots\varphi_{W_n}$ be a word in $\Phi^*$. We define the vector $\kappa_A\in\Set{0,1}^n$ attached to $A$ (i.e., $\kappa_A$ is vector of length $n$ with zero/one entries). The entries of $\kappa_A$ are defined as follows: the $i$-th coordinate of $\kappa_A$ is \emph{one} if and only if there are decompositions $W_{i-1}=W_{i-1}'W_{i-1}''$ and $W_{i+1}=W_{i+1}'W_{i+1}''$ such that $W_{i-1}'' W_i W_{i+1}' \in \R$ and $\lp{ W_{i-1}'' W_i W_{i+1}' } > \lp{ \Comp{W_{i-1}'' W_i W_{i+1}'} }$. To complete the definition we need to define $W_0$ and $W_{n+1}$ so we set $W_0=W_{n+1}=\eps$ (where $\eps$ is the empty word).
\end{definition}

To give some intuition, the vector $\kappa_A$ marks these points in $\eta(A)$ that are ``inefficient'' in the number of pieces.  We next define an order on $\Phi^*$ which is based the auxiliary vectors.

\begin{definition}[Piefer order ``$\prec$'']
Let $A$ and $B$ be two elements of $\Phi^*$. We write $A \prec B$ (read: `$A$ precedes $B$ in the Piefer order') if $\kappa_A=\kappa_B$ or $\kappa_A$ precedes $\kappa_B$ in shortlex order.
\end{definition}

Note that, for example, if $|A|<|B|$ then $A \prec B$. Note also that the order ``$\prec$'' is regular. An important property of the order ``$\prec$'' is that for any $s\in S$ there is a ``$\prec$''-minimal element $A$ such that $A$ presents $s$ (see the paragraph after Theorem \ref{thm:falseification} for the definition of ``$\prec$''-minimal). This follows from the fact that shortlex ordering is a well ordering.

The proof of the main theorem will be completed if we establish that conditions (R) and (FT) of Theorem \ref{thm:falseification} hold for the order ``$\prec$''. The proof of these two properties occupies the rest of this section. Consider an element $A \in \Phi^*$ that is not minimal according to the order ``$\prec$''. Suppose another element $B \in \Phi^*$ has the following three properties:
\begin{enumerate}
	\item $\pi(A)=\pi(B)$
	\item $B \prec A$
	\item $A$ and $B$ are $k$-fellow-travellers.
\end{enumerate}
In this case, following the terminology in \cite{Pei96}, we will say that ``\emph{$B$ $k$-refute $A$}''. To show condition (R) we need to show that for any element $A$ that is not minimal according to the order ``$\prec$'' we have some element $B$ that $k$-refute $A$. We start with two basic definitions for words over $\Phi$:

\begin{definition}[Admissible] \label{def:admissible}
We say that $A=\varphi_{W_1}\cdots\varphi_{W_n} \in \Phi^*$ is admissible if for all $1 \leq i < n$ we have $W_i W_{i+1} \notin \Sb$.
\end{definition}

\begin{definition}[Efficient Words in $\Phi^*$] \label{def:efficient}
We say that $A \in \Phi^*$ is \emph{efficient} if it is admissible and $\kappa_A$ is a zero vector (i.e., all its entries are zero). Elements of $\Phi^*$ which are not efficient will be called \emph{inefficient}.
\end{definition}

To distinguish between zero and non-zero vectors so we adopt the notation $\kappa_A=\overline{0}$ to denote that $\kappa_A$ is a zero vector (of some length) and $\kappa_A \neq \overline{0}$ when $\kappa_A$ is not all zeros. Also, to refer to the coordinates of the vector $\kappa_A$ we will use the notations $\cord{\kappa_A}{i}$ which will denote the $i$-th coordinate of the vector.

A technical observation is that condition (R) holds for all inefficient elements of $\Phi^*$. This is stated in the following proposition. 

\begin{proposition} \label{prop:inefficientCanRefute}
Let $A \in \Phi^*$. If $A$ is inefficient then $A$ can be $3$-refuted.
\end{proposition}

Here is the proof of Proposition \ref{prop:inefficientCanRefute} for the case that $A$ is not admissible. The rest of the proof of the proposition is left to the next sub-section.

\begin{lemma}\label{lem:adjoingWhenInS}
Suppose $A=\varphi_{W_1}\cdots\varphi_{W_n} \in \Phi^*$ is not admissible. Then, there is an element $B$ that $1$-refutes $A$. Moreover, we have that $|B|<|A|$ and $\eta(A)=\eta(B)$.
\end{lemma}
\begin{proof}
Since $A$ is not admissible there is an index $1\leq \ell < n$ such that $W_\ell W_{\ell+1} \in \Sb$. Construct $B$ from $A$ by replacing the two consecutive generators $\varphi_{W_{\ell-1}} \varphi_{W_{\ell}}$ with the generator $\varphi_{W_{\ell-1} W_{\ell}}$. Namely, 
\[
B=\varphi_{W_1} \cdots \varphi_{W_{\ell-2}} \, \varphi_{W_{\ell-1} W_{\ell}} \, \varphi_{W_{\ell+1}} \cdots \varphi_{W_n}
\]
Then, $|B|<|A|$ implying that $B\prec A$. Clearly we have that $\eta(B)=\eta(A)$ and so $\pi(A)=\pi(B)$. We finish by showing that $A$ and $B$ are $1$-fellow-travellers. Recall that $A(j)$ denotes the prefix of $A$ of length $j$, that $\pi(C)$ is the element in $S$ presented by $C$, and that $d_\Phi$ is the induced metric on the Cayley graph. Since $\pi(B(j)) = \pi(A(j))$ for all $1\leq j\leq \ell-2$ and $\pi(B(j))=\pi(A(j) \varphi_{W_{j+1}})$ for all $\ell-1\leq j\leq n-1$ we get that $d_\Phi(A(j),B(j))\leq1$ for all $1 \leq j \leq n$ so $B$ and $A$ are $1$-fellow-travellers.
\end{proof}

The next lemma shows how to check for inefficiency.

\begin{lemma} \label{lem:whenKappaNotZero}
Let $A \in \Phi^*$ be admissible. Then, $A$ is inefficient if and only if $\eta(A)$ contains a subword $W \in \R$ such that $\lp{W} > \lp{\Comp{W}}$.
\end{lemma}
\begin{proof}
The `if' part follows from the definition of $\kappa_A$. We prove the `only if' part. Suppose $A=\varphi_{W_1}\cdots\varphi_{W_n}$. Suppose further that $\eta(A) = V_1 L V_2$ where $L \in \R$ and $L=R$ is a defining relation with the property that $\lp{L} > \lp{R}$ ($R$ is the complement of $L$). By the (\dag) condition we have that $\lp{L}\geq4$. Let $k$ be the smallest index such that $V_1$ is a prefix of $W_1 W_2 \cdots W_{k}$ (which is equal to $\eta(A(k))$). We are done if $V_1 = W_1 W_2 \cdots W_{k}$ because then $W_{k+1}$ is a prefix of $L$ and thus $\cord{\kappa_A}{k+1} = 1$. Otherwise, let $W_{k} = W_{k}' W_{k}''$ where $V_1 = W_1 W_2 \cdots W_{k-1}W_{k}'$ and both $W_{k}'$ and $W_{k}''$ are not empty. Decompose $L$ as $L = W_{k}'' T$. If $W_{k+1}$ is a prefix of $T$ then as above we get that $\cord{\kappa_A}{k+1} = 1$. Thus, we can assume that $T$ is a prefix of $W_{k+1}$. In this case, $\lp{W_{k}''} \leq 1$ (follows from Observation \ref{obs:sharedPieceLP1} because $W_{k}''$ is a subword of $W_k$ and a prefix of $L$) and thus $\lp{T} \geq 3$. Now, $T$ is a suffix of $L$ so we must have that $T = W_{k+1}$ (follows from Observation \ref{obs:twoPiecUniq} since $L$ is the unique element in $\R$ that $T$ is its subword). This implies the lemma since we now have that $\cord{\kappa_A}{k+1} = 1$.
\end{proof}

The following corollaries are immediate from Lemma \ref{lem:whenKappaNotZero}.

\begin{corollary} \label{cor:NonZeroKappaPropOfEtaA}
Suppose $A,B \in \Phi^*$ with $\eta(A)=\eta(B)$ and both admissible. Then, $A$ is efficient if and only if $B$ is efficient.
\end{corollary}
\begin{proof}
By Lemma \ref{lem:whenKappaNotZero} is enough to know $\eta(A)$ to know if $A$ is efficient.
\end{proof}

\begin{corollary} \label{cor:condDdagHolds}
Suppose $A,B \in \Phi^*$ are efficient and $M$ is a $(\mu\xi,\sigma)$-thin diagram with $\eta(A)$ the label of $\mu$ and $\eta(B)$ the label of $\sigma$. Suppose further that $\xi$ is empty or is labelled with an element of $X$. Then, condition \emph{(\ddag)} of Lemma \ref{lem:essencialVert} holds for $M$.
\end{corollary}
\begin{proof}
Assume there is a boundary region $D$ in $M$ with boundary $\delta\rho^{-1}$ such that $\delta$ and $\rho$ have positive labels $V_\delta$ and $V_\rho$ and $\lp{V_\delta}>\lp{V_\rho}$ (recall that $V_\delta$ and $V_\rho$ are defining words and are complements of each other). If $\delta$ is a subpath of $\mu$ then it follows from Lemma \ref{lem:whenKappaNotZero} that $A$ is inefficient but that would contradict the assumption. Hence, $\delta$ is not a subpath of $\mu$. Similarly, $\delta$ is not a subpath of $\sigma$. This shows that condition (\ddag) holds.
\end{proof}

\begin{corollary} \label{cor:bigPartToNonZeroKappa}
Let $A \in \Phi^*$ and let $M$ be a semigroup diagram with a boundary path $\mu$ labelled by $\eta(A)$. Assume there is a boundary region $D$ in $M$ with boundary $\rho\delta^{-1}$ such that: (1) $\rho$ and $\delta$ have positive labels $V_\rho$ and $V_\delta$; (2) $\rho$ is the outer boundary of $D$ and is a subpath of $\mu$; (3) $\lp{V_\delta} \leq 3$. Then, $A$ is inefficient.
\end{corollary}
\begin{proof}
It follows from Definition \ref{def:efficient} that we can assume that $A$ is admissible. The rest follows from Lemma \ref{lem:whenKappaNotZero} since $V_\rho = V_\delta$ is a defining relation and by condition (\dag) of the main theorem we have that $\lp{V_\rho} \geq 4$ so $V_\rho$ is a subword of $\eta(A)$ with  $\lp{V_\rho} > \lp{V_\delta} = \lp{\Comp{V_\rho}}$.
\end{proof}

We continue with a lemma which makes the connection to the diagrams of $S$.

\begin{lemma} \label{lem:diagThinWhenKapNull}
Let $A,B \in \Phi^*$ be efficient. Suppose there is an element $x \in X\cup\Set{\eps}$ such that $\eta(A) x = \eta(B)$ in $S$. Suppose further that $M$ is a minimal diagram with boundary $\mu\xi\sigma^{-1}$ such that $\mu$ is labelled by $\eta(A)$, $\xi$ is labelled by $x$, and $\sigma$ is labelled by $\eta(B)$. Then, $M$ is a $(\mu\xi,\sigma)$-thin diagram.
\end{lemma}
\begin{proof}
Assume by contradiction that $M$ is not $(\mu\xi,\sigma)$-thin. By Proposition \ref{prop:condForThin} we have,
without loss of generality, a region $D$ with the following properties:
\begin{enumerate}
 \item $\partial D = \rho \delta^{-1}$ where $\rho$ and $\delta$ have positive labels, $\rho$ is the outter boundary of $D$, and $\delta$ is the inner boundary of $D$.
 \item if $x \neq \eps$ then $\rho$ does not contain $\xi$.
 \item $N_\delta^D = 3$ (the number of neighbors of $D$ along $\delta$ is $3$).
\end{enumerate}
By the second part we have that $\rho$, the outer boundary of $D$, is a subpath of $\mu$ or $\sigma$. Thus, by Corollary \ref{cor:bigPartToNonZeroKappa} we get that $A$ is inefficient which is a contradiction to the assumption on $A$.
\end{proof}

Using the above lemma we can prove the following technical proposition. A similar result in the context of groups is straightforward. It turns out that for semigroups one must work a little bit harder.

\begin{proposition} \label{prop:zeroKappaToKFT}
Let $A,B \in \Phi^*$ be efficient. Suppose that $\pi(A)=\pi(B)$ or there is an element $x\in X$ such that $\pi(A \varphi_x) = \pi(B)$ in $S$. Suppose further that $B$ is a geodesic. Then, one of the following options hold:
\begin{enumerate}
 \item $A$ and $B$ are $3$-fellow-travellers.
 \item There is an element $C \in \Phi^*$ such that $\pi(A) = \pi(C)$, $|C|<|A|$, and $C$ and $A$ are $2$-fellow-travellers (namely, $C$ $2$-refute $A$).
\end{enumerate}
\end{proposition}

Before we give the (rather long) proof of the proposition we show that it implies that the conditions of Theorem \ref{thm:falseification}.

\begin{corollary}
Conditions (R) and (FT) of Theorem \ref{thm:falseification} hold for the order ``$\prec$''. In particular, $S$ is an automatic semigroup.
\end{corollary}
\begin{proof}
First we prove that condition (R) holds. Let $A\in\Phi^*$ be an element that is \emph{not} ``$\prec$''-minimal element. By Proposition \ref{prop:inefficientCanRefute} we can assume that $A$ is efficient (or otherwise it can be $3$-refuted by the proposition). Take $B\in\Phi^*$ such that $\pi(A)=\pi(B)$ and $B$ is ``$\prec$''-minimal. By the same proposition we get that $B$ is efficient. By minimality, $B$ is a geodesic. By Proposition \ref{prop:zeroKappaToKFT} either $A$ and $B$ are $3$-fellow-travellers and thus $A$ is $3$-refuted by $B$ or there is an element $C$ that $2$-refutes $A$. This shows that condition (R) holds. Next we prove condition (FT). Let $k$ be a bound on the lengths of the elements in $\Sb$ (recall that $\Sb$ is a finite set). Take two ``$\prec$''-minimal elements $A$ and $B$ such that $d_\Phi(A,B)\leq 1$. As above, $A$ and $B$ are efficient and geodesics. If we have that $d_\Phi(A,B) = 0$ then $\pi(A) = \pi(B)$ so by Proposition \ref{prop:zeroKappaToKFT} we have that $A$ and $B$ are $3$-fellow-travellers. Assume that $d_\Phi(A,B) = 1$. Then, there is $\varphi_V \in \Phi$ such that, switching $A$ and $B$ if necessary, $\pi(A \varphi_V) = \pi(B)$. We claim that $A$ and $B$ are $3k$-fellow-travellers. Let $V=x_1 x_2 \cdots x_n$. Take elements $C_0,C_1,C_2,\ldots,C_{n-1},C_n$ in $\Phi^*$ such that each $C_j$ is ``$\prec$''-minimal, $C_0=A$ and $\pi(C_j) = \pi(C_{j-1} \varphi_{x_j})$. We have that $\pi(C_{n}) = \pi(A \varphi_{x_1}\cdots \varphi_{x_n}) = \pi(A \varphi_{x_1 \cdots x_n}) = \pi(A \varphi_V) = \pi(B)$ and thus we take $C_n$ to be $B$. By Proposition \ref{prop:zeroKappaToKFT} we get that $C_{j-1}$ and $C_j$ are $3$-fellow-travellers (by ``$\prec$''-minimality and Proposition \ref{prop:inefficientCanRefute} both are efficient). Consequently, $C_0$ and $C_n$ are $3n$-fellow-travellers (follows from Lemma \ref{lem:twoKFTs}). Finally, because $n\leq k$ we get that $A$ and $B$ are $3k$-fellow-travellers.
\end{proof}

The rest of the section is devoted to the proof of Proposition \ref{prop:zeroKappaToKFT}. An element $A \in \Phi^*$ is called \emph{semi-geodesic} if for any other $B \in \Phi^*$ such that $\eta(A)=\eta(B)$ (equality as elements of $X^*$) we have $|A|\leq|B|$. Obviously, a geodesic is also a semi-geodesic but the converse may not be true. 

\begin{lemma}\label{lem:semiGeoToFT}
Let $A,B\in\Phi^*$ and suppose $\eta(A)x = \eta(B)$ for some $x \in X \cup \Set{\eps}$. If $B$ is semi-geodesic then either:
\begin{enumerate}
 \item $A$ and $B$ are $1$-fellow-travellers; or,
 \item there is $C \in \Phi^*$ such that $|C|<|A|$, $\eta(A)=\eta(C)$, and $A$ and $C$ are $2$-fellow-travellers.
\end{enumerate}
\end{lemma}
\begin{proof}
By Lemma \ref{lem:adjoingWhenInS} we can assume that $A$ is admissible since otherwise the second case holds (i.e., there are no two consecutive letters $\varphi_{W_i} \varphi_{W_{i+1}}$ in $A$ such that $W_i W_{i+1} \in \Sb$). For every two indexes $r$ and $s$ there is a an element $V_{r,s}\in X^*$ such that either $\eta(A(r)) V_{r,s} =
\eta(B(s))$ or $\eta(A(r)) = \eta(B(s)) V_{r,s}$. Let $D(r,s)$ denote the minimal $k$ such that $V_{r,s}$ decomposes as $V_{r,s} = U_1 U_2 \cdots U_k$ and $U_i \in \Sb$ for $1 \leq i \leq k$ (it is zero if $V_{r,s} = \eps$). If $D(r,r)\leq 1$ for all $r$ then clearly $A$ and $B$ are $1$-fellow-travellers. Otherwise, let $r$ be the minimal index such $D(r,r)\geq2$. Notice that $|\eta(A(r))| < |\eta(B(r))|$ by the fact that $B$ is a semi-geodesic. (If not, take a minimal index $s\geq r$ such that $|\eta(B(s))| \geq |\eta(A(r))|$. Then, $s - r \geq 2$ and $D(r,s)\leq 1$ so we can replace the prefix $B(s)$ of $B$ with $A(r)\varphi_U$ for some $U\in\Sb$ and thus reduce the length of $B$ which is impossible if $B$ is a semi-geodesic.) Next, take a maximal index $s$ such that $|\eta(B(s))| \leq |\eta(A(r))|$ so $r-s \geq 2$ and $D(r,s)\leq1$. Let $U = V_{r,s}$. By the fact that $D(r,s)\leq1$ we have that $U=\eps$ or $U \in \Sb$. Suppose $A=A(r) T$ and let $C \in \Phi^*$ such that $C = B(s) T$ if $U=\eps$ and $C = B(s) \varphi_U T$ otherwise. Then, $\eta(C) = \eta(A)$. Also, $|C| < |A|$ by the following computation:
\[
\begin{array}{rl}
  |C| & \leq |B(s)| + 1 + |T| \\
  & = s + 1 + |T| \\
  & < r + |T| = |A(r)| + |T| = |A|
\end{array}
\]
We claim that $A$ and $C$ are $2$-fellow-travellers. For simplicity we will prove this under the assumption that $U\neq\eps$. By minimality of $r$ we have that $A(s)$ and $C(s)$ are $1$-fellow-travellers. By the construction of $C$ we have that, $\eta(C(s+1))=\eta(A(r))$. Hence, it is enough to show that $r-s \leq 3$. Let $A(r)=A(s) Q$ and suppose $A = \varphi_{W_1} \cdots \varphi_{W_n}$ then $\eta(Q)=W_{s+1} W_{s+2} \cdots W_r$. We have that $D(s,s)\leq 1$ and $D(r,s)\leq1$ so we have a decomposition of $\eta(Q)$ into $U_1 U_2 \cdots U_k$ where $U_i \in B$ for $1 \leq i \leq k$ and $k\leq2$. Consequently, by the pigeonhole principle if $r-s > 3$ there will be an index $s+1 \leq i < r$ such that $W_i W_{i+1}$ is an element of $\Sb$. This contradicts our assumption that $A$ is admissible.
\end{proof}

We say that $A = \varphi_{W_1} \cdots \varphi_{W_n}$ in $\Phi^*$ is \emph{left-greedy} if for every $1 \leq i < n$ we have that $W_i x \not\in \Sb$ where $x$ is the first letter of $W_{i+1}$. Clearly, for any $A \in \Phi^*$ there is only one left-greedy representative $A'$ such that $\eta(A)=\eta(A')$.

\begin{lemma} \label{lem:leftGreedySemiGeodesic}
For every $A \in \Phi^*$ there is a unique element $A'\in\Phi^*$ such that $A'$ is semi-geodesic, left-greedy, and $\eta(A) = \eta(A')$.
\end{lemma}
\begin{proof}
Let $n$ be the length of a semi-geodesic $B$ such that $\eta(A)=\eta(B)$. We prove the lemma by induction on $n$. If $n=1$ then there is nothing to prove. Suppose the lemma holds for $n-1$. Let $B = \varphi_{U_1} \cdots \varphi_{U_n}$ be a semi-geodesic as above. Consider the decomposition $U_1 U_2 = V_1 V_2$ where $V_1$ is the maximal prefix of $U_1 U_2$ such that $V_1\in\Sb$. Set $C = \varphi_{V_2} \varphi_{U_3} \cdots \varphi_{U_n}$. Then, $C$ is semi-geodesic and $|C|=n-1$. By induction, there is some $C'$ that is semi-geodesic, left-greedy, and $\eta(C)=\eta(C')$. Thus, $A' = \varphi_{V_1} C'$ is semi-geodesic, left-greedy, and $\eta(A)=\eta(A')$, as needed.
\end{proof}

By the uniqueness of the left-greedy representative, it follows from the above lemma that if $A \in \Phi^*$ is left-greedy then it is necessarily semi-geodesic. This fact is used in next lemma to check that a given element is semi-geodesic. It is useful to remember that $A = \varphi_{W_1}\cdots\varphi_{W_n}$ is left-greedy if and only $\varphi_{W_i}\varphi_{W_{i+1}}$ is left-greedy for all $1\leq i <n$.

\begin{lemma} \label{lem:multLeftGredElmts}
Let $A_0,\ldots,A_n$ and $B_1,\ldots,B_n$ be left-greedy elements of $\Phi^*$. Assume the following:
\begin{enumerate}
 \item If $\varphi_V$ is the first letter of $B_i$ then $\lp{V}\geq3$ and $V$ is a prefix of some element in $\R$.
 \item If $\varphi_V$ is the last letter of $B_i$ then $\lp{V}\geq3$ and $V$ is a suffix of some element in $\R$.
\end{enumerate}
Then, the element $C = A_0 B_1 A_1 \cdots B_n A_n$ is semi-geodesic.
\end{lemma}
\begin{proof}
The proof is mostly routine. It follows for the conditions above that if $\varphi_{W_j}\varphi_{W_{j+1}}$ (where $1 \leq j < |C|$) are two consecutive letters in $C$ which are \emph{not} left-greedy then there is an index $i$ such that one of the following holds:
\begin{enumerate}
 \item $\varphi_{W_j}$ is the last letter of $B_i$ and $\varphi_{W_{j+1}}$ is the first letter of $A_i$. 
 \item $\varphi_{W_j}$ is the last letter of $A_i$ and $\varphi_{W_{j+1}}$ is the first letter of $B_{i+1}$. 
\end{enumerate}
The first case is impossible since we have that $\lp{W_j}\geq3$ and $W_j$ is a suffix of some unique element in $\R$ (uniqueness follows from Observation \ref{obs:twoPiecUniq}). Hence, $W_{j}$ is not a proper prefix of any element in $\Sb$ and thus $\varphi_{W_j} \varphi_{W_{j+1}}$ must be left-greedy. So, we are left with the second case. Since only the second case is possible it follows that if $\varphi_{W_j}\varphi_{W_{j+1}}$ is not left greedy and also $\varphi_{W_k}\varphi_{W_{k+1}}$ is not left greedy for some $1\leq j < k < n$ then $k\geq j+2$ (i.e., there is no overlap between the two pairs and moreover there is a gap of at least one letter between them). Consequently, if we can show that we can replace each pair of consecutive letters $\varphi_{W_j} \varphi_{W_{j+1}}$ with a \emph{left-greedy} pair $\varphi_{U_j} \varphi_{U_{j+1}}$ such that $W_j W_{j+1} = U_j U_{j+1}$ and also $\varphi_{U_{j+1}} \varphi_{W_{j+2}}$ is left-greedy then we can `fix' $C$ so it becomes left-greedy and thus we would show that $C$ is semi-geodesic. So, suppose we fix $\varphi_{W_j}\varphi_{W_{j+1}}$ into a left-greedy $\varphi_{U_j}\varphi_{U_{j+1}}$ such that $W_j W_{j+1} = U_j U_{j+1}$. This induces a decomposition of $W_{j+1} = W_{j+1}' W_{j+1}''$ where $W_{j+1}''=U_{j+1}$. Since $\lp{W_{j+1}}\geq3$ we have by Observation \ref{obs:sharedPieceLP1} that the piece-length of $W_{j+1}'$ is at most one and thus $\lp{U_{j+1}}\geq2$. Therefore by Observation \ref{obs:twoPiecUniq} there are unique $V\in\Sb$ and $R\in\R$ such that $U_{j+1}V$ is a suffix of $R$. Because $U_{j+1}$ is a suffix of $W_{j+1}$ we get that also $W_2V$ is a suffix of $R$. Now, $\varphi_{W_{j+1}}\varphi_{W_{j+2}}$ is left-greedy so also $\varphi_{U_{j+1}}\varphi_{W_{j+2}}$ is left greedy.
\end{proof}

Let $M$ be a $(\mu\xi,\sigma)$-thin diagram where $\xi$ is labelled by an element of $X \cup \Set{\eps}$. A \emph{fundamental decomposition} of $M$ is a decomposition $M = \rho_0 \cup M_1 \cup \rho_1 \cup \cdots \cup \rho_{k-1} \cup M_k \cup \rho_k$ such that $M_1, M_2, \ldots, M_k$ are the connected components of the closure of the interior of $M$ and $\rho_0, \rho_1, \ldots, \rho_k$ are the paths in the closure of $M \setminus \left( \cup_{j=1}^{k}M_k \right)$. The path $\rho_i$ connects $M_i$ to $M_{i+1}$ for $1\leq i < k$. The paths $\rho_0$ and $\rho_k$ my be empty (and in this case we think on them as being a single vertex) or otherwise only intersects $M_1$ and $M_k$, respectively. See Figure \ref{fig:fundDecomp}. The fundamental decomposition induces the definition of two elements as defined next.

\begin{figure}[ht]
\centering
\includegraphics[totalheight=0.11\textheight]{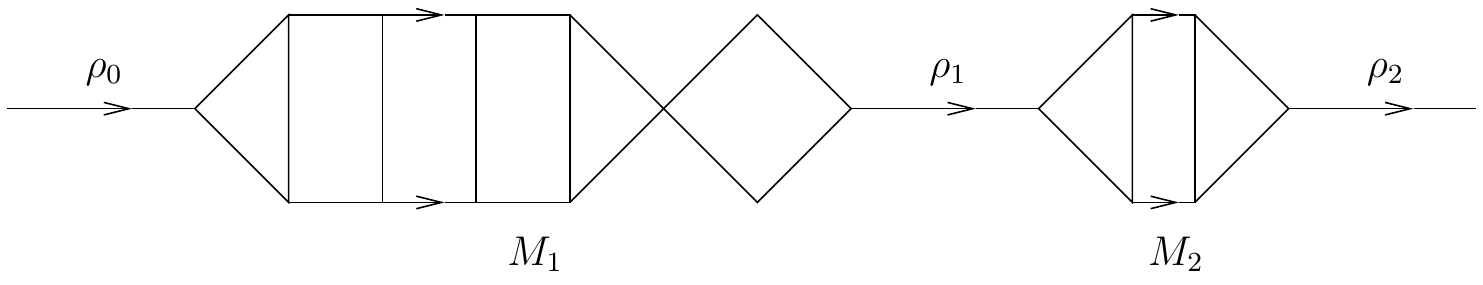}
\caption{Illustration of fundamental decomposition} \label{fig:fundDecomp}
\end{figure}

\begin{definition} \label{def:CMuSigma}
Let $M$ be a $(\mu\xi,\sigma)$-thin diagram where $\xi$ is labelled by an element of $X \cup \Set{\eps}$ and consider a fundamental decomposition $M = \rho_0 \cup M_1 \cup \rho_1 \cup \cdots \cup \rho_{k-1} \cup M_k \cup \rho_k$. We define two elements, $C_\mu$ and $C_\sigma$, of $\Phi^*$. $C_\mu$ is defined to be $C_\mu = C_{\rho_0} C^{\mu}_{M_1} C_{\rho_1} \cdots C_{\rho_{k-1}} C^{\mu}_{M_k} C_{\rho_k}$ where:
\begin{enumerate}
 \item $C_{\rho_j}$ is the left-greedy semi-geodesic element such that $\eta(C_{\rho_j})$ is the label of $\rho_j \cap \mu$
 \item Suppose that in $M_j$ the regions $D^j_1,D^j_2,\ldots,D^j_{N_j}$ have the property that $\partial D^j_i \cap \mu$ contains an edge for $1 \leq i \leq N_j$. We assume that the indexing of the regions corresponds to the order they intersect with $\mu$. We define $C^{\mu}_{M_j}$ to be the element $\varphi_{V_1} \cdots \varphi_{V_{N_J}}$ where $V_i$ is the label of $\partial D^j_i \cap \mu$.
\end{enumerate}
$C_\sigma$ is defined similarly by replacing $\mu$ with $\sigma$.
\end{definition}

In the next lemma we analyze the properties of the elements $C_\mu$ and $C_\sigma$ from Definition \ref{def:CMuSigma}. We will use the results of Lemma \ref{lem:essencialVert} for the proof.

\begin{lemma} \label{lem:propOfCMuSigma}
Let the notation be as in Definition \ref{def:CMuSigma} and assume that condition (\ddag) of Lemma \ref{lem:essencialVert} holds for the diagram $M$. Let $1\leq i \leq k$. Then:
\begin{enumerate}
	\item \label{lem:propOfCMuSigma:eq} $|C^{\mu}_{M_i}| = |C^{\sigma}_{M_i}|$. Moreover, suppose that $A$ is a prefix of $C^{\mu}_{M_i}$ of length $n$ and $B$ is a prefix of $C^{\sigma}_{M_i}$ of length $n$. Suppose further that $\delta$ and $\rho$ are subpaths of $\partial M_i$ which are labelled by $\eta(A)$ and $\eta(B)$, respectively. Then, the terminal vertices of $\delta$ and $\rho$ belong to the boundary of the same region in $M_i$.
	\item Suppose $i \neq k$. Then, if $\varphi_V$ is the first (resp., the last) letter of $C^{\mu}_{M_i}$ or $C^{\sigma}_{M_i}$ then $\lp{V}\geq3$. For $i=k$ the assertion holds for the first letter and holds for the last letter if $\rho_k$ is not empty.
	\item $C^{\mu}_{M_i}$ and $C^{\sigma}_{M_i}$ are left-greedy.
\end{enumerate}
\end{lemma}
\begin{proof}
We use the notation of Definition \ref{def:CMuSigma}. Since we assumed that condition (\ddag) of Lemma \ref{lem:essencialVert} holds we can use its conclusions. By Part \ref{lem:essencialVert:vt} of Lemma \ref{lem:essencialVert} each region $D$ in $M$ has the property that $\partial D \cap \mu$ and $\partial D \cap \sigma$ contain an edge. This shows that if $D_1,\ldots,D_n$ are the regions of $M_i$ then by construction $C^{\mu}_{M_i} = \varphi_{V_1} \cdots \varphi_{V_n}$ and $V_j$ is the label of $\partial D_j \cap \mu$ for $1\leq j \leq n$. Consequently, $|C^{\mu}_{M_i}| = n$ and similarly $|C^{\sigma}_{M_i}| = n$ so we get the first part of the lemma (the `moreover' part follows along the same lines). By Part \ref{lem:essencialVert:br} of Lemma \ref{lem:essencialVert} we get that $\lp{V_1}\geq3$ and $\lp{V_n}\geq3$ when $i\neq k$. For $i=k$ the same holds for $V_1$ and it holds for $V_n$ if $\partial D_n \cap \mu$ does not contain $\xi$ which is the case if $\rho_k$ is not empty. This proves the second part of the lemma. Finally, by Part \ref{lem:essencialVert:gr} of Lemma \ref{lem:essencialVert} we get that $\varphi_{V_j}\varphi_{V_{j+1}}$ is left greedy for $1\leq j < k$ so consequently $C^{\mu}_{M_i}$ is left greedy. This proves the last part of the lemma.
\end{proof}

\begin{corollary} \label{cor:indElemIsSemiGeo}
Let the notation be as in Definition \ref{def:CMuSigma} and assume that condition (\ddag) of Lemma \ref{lem:essencialVert} holds for the diagram $M$. Then, $C_\mu$ and $C_\sigma$ are semi-geodesics.
\end{corollary}
\begin{proof}
We prove the corollary for $C_\mu$; the proof for $C_\sigma$ is similar. First notice that by the definition of $C^{\mu}_{M_i}$ above we get that it starts with a letter $\varphi_{V_1}$ such that $V_1$ is a prefix of some element in $\R$. Also, $C^{\mu}_{M_i}$ ends with a letter $\varphi_{V_{N_k}}$ such that $V_{N_k}$ is a suffix of some element in $\R$. It follows from Lemma \ref{lem:propOfCMuSigma} that the conditions of Lemma \ref{lem:multLeftGredElmts} holds when we set $A_i = C_{\rho_i}$ and $B_i = C^{\mu}_{M_i}$. Thus, $C_\mu = A_0 B_1 A_1 \cdots B_n A_n$ is semi-geodesic. 
\end{proof}

\begin{lemma}\label{lem:WmuAndWsigAre1FT}
Let the notation be as in Definition \ref{def:CMuSigma} and assume that condition (\ddag) of Lemma \ref{lem:essencialVert} holds for the diagram $M$. Then, $C_\mu$ and $C_\sigma$ are $1$-fellow-travellers.
\end{lemma}
\begin{proof}
We need to show that $d_\Phi(C_\mu(n),C_\sigma(n)) \leq 1$ for all $n$. This follows routinely from the construction of $C_\mu$ and $C_\sigma$. By Part \ref{lem:propOfCMuSigma:eq} of Lemma \ref{lem:propOfCMuSigma} we have that $|C^{\mu}_{M_j}| = |C^{\sigma}_{M_j}|$ for all $j$. Thus, again by Part \ref{lem:propOfCMuSigma:eq} of Lemma \ref{lem:propOfCMuSigma}, the subpaths of $\mu$ and $\sigma$ labelled by $\eta(C_\mu(n))$ and $\eta(C_\sigma(n))$, respectively, terminate at the same vertex (if it belongs to some $\rho_j$) or at vertices that belong to the boundary of the same region. Consequently, we have $V \in \Sb\cup\Set{\eps}$ such that either $\eta(C_\mu(n)) V = \eta(C_\sigma(n))$ or $\eta(C_\mu(n)) = \eta(C_\sigma(n)) V$. So, $\pi(C_\sigma(n)\varphi_V) = \pi(C_\sigma(n))$ or $\pi(C_\sigma(n)) = \pi(C_\sigma(n)\varphi_V)$. Consequently, $d_\Phi(C_\mu(n),C_\sigma(n)) \leq 1$.
\end{proof}

\begin{proof}[Proof of Proposition \ref{prop:zeroKappaToKFT}]

Suppose we are given two element $A,B \in \Phi^*$ such that the conditions of Proposition \ref{prop:zeroKappaToKFT} hold. Let $M$ be a van Kampen diagram with boundary path $\mu\xi\sigma^{-1}$ such that $\mu$ is labelled by $\eta(A)$, $\xi$ is empty or is labelled by some $x\in X$, and $\sigma$ is labelled by $\eta(B)$. Then, $M$ is $(\mu\xi,\sigma)$-thin diagram by Lemma \ref{lem:diagThinWhenKapNull}. By Corollary \ref{cor:condDdagHolds} we have that condition (\ddag) of Lemma \ref{lem:essencialVert} holds for $M$. Let $C_\mu$ and $C_\sigma$ be the elements induced from the fundamental decomposition (Definition \ref{def:CMuSigma}) of $M$ and for which $\eta(A)=\eta(C_\mu)$ and $\eta(B)=\eta(C_\sigma)$. By Corollary \ref{cor:indElemIsSemiGeo} we have that $C_\mu$ and $C_\sigma$ are semi-geodesic (which also implies that they are admissible). Therefore, it follows by Corollary \ref{cor:NonZeroKappaPropOfEtaA} that $C_\mu$ and $C_\sigma$ are efficient. Thus, by Lemma \ref{lem:semiGeoToFT} we have that $B$ and $C_\sigma$ are $1$-fellow-travellers (recall that $B$ is geodesic and $C_\sigma$ is semi-geodesic). By the same lemma, either $A$ and $C_{\mu}$ are $1$-fellow-travellers or we have $A'$ that $2$-refutes $A$ so the proposition is satisfied. Hence, we can assume that $A$ and $C_\mu$ are $1$-fellow-travellers. Using Lemma \ref{lem:WmuAndWsigAre1FT} we have that $C_\mu$ and $C_\sigma$ are $1$-fellow-travellers and thus $A$ and $B$ are $3$-fellow-travellers (using Lemma \ref{lem:twoKFTs}) which proves the proposition.
\end{proof}

\subsection{Refuting Inefficient Elements}

In this sub-section we prove Proposition \ref{prop:inefficientCanRefute}. Recall that by Lemma \ref{lem:adjoingWhenInS} the proposition follows for non-admissible elements (see Definition \ref{def:admissible}). Thus, the main difficulty is to prove Proposition \ref{prop:inefficientCanRefute} for admissible elements. We introduce the following definition, which allows us to prove Proposition \ref{prop:inefficientCanRefute} by induction.

\begin{definition}[$\ell$-Pacing Pair] \label{def:pacingPair}
Let $A=\varphi_{W_1}\cdots\varphi_{W_n}$ and $B=\varphi_{U_1}\cdots\varphi_{U_m}$ be in $\Phi^*$ such that $\kappa_A \neq \overline{0}$ and let $1 \leq \ell \leq |A|$. We say that $(A,B)$ is an \emph{$\ell$-pacing pair} if the following conditions hold:
\begin{description}
 \item \texttt{C1.} $A$ is admissible and $\cord{\kappa_A}{\ell}=1$.

 \item \texttt{C2.} $\pi(A) = \pi(B)$ and $|A| = |B|$.

 \item \texttt{C3.} There is an index $1\leq j\leq\ell$ such that:
 \begin{enumerate}
  \item $d_\Phi(A(i),B(i))=0$ for all $i<j$ and $i>\ell$.
  \item $d_\Phi(A(\ell),B(\ell))\leq1$.
  \item $d_\Phi(A(i),B(i))\leq2$ for all $j\leq i\leq\ell-1$.
 \end{enumerate}

 \item \texttt{C4.} If $\ell<n$ then $U_{\ell+1}$ is a subword of $W_{\ell+1}$ such that $\lp{U_{\ell+1}} \geq \lp{W_{\ell+1}} - 1$ and $U_{j} = W_{j}$ for all $\ell+2 \leq j \leq n$.

 \item \texttt{C5.} Either $B$ is not admissible or:
 \begin{enumerate}
  \item $\cord{\kappa_B}{\ell}<\cord{\kappa_A}{\ell}$
  \item For all $1\leq i\leq \ell-1$ we have $\cord{\kappa_B}{i} \leq \cord{\kappa_A}{i}$.
 \end{enumerate}
\end{description}
We say that $(A,B)$ is a \emph{pacing pair} if it is an $\ell$-pacing pair for some $\ell$. 
\end{definition}

First we show that if $(A,B)$ is a pacing pair then $A$ can be refuted. 

\begin{lemma}
Let $A \in \Phi^*$ be a admissible but not efficient and suppose $(A,B)$ is a pacing pair. Then, $A$ can be $3$-refuted.
\end{lemma}
\begin{proof}
We use conditions \texttt{C2}, \texttt{C3}, and \texttt{C4} of Definition \ref{def:pacingPair}. By \texttt{C2} we have that $\pi(A) = \pi(B)$. By \texttt{C3} we have that $A$ and $B$ are $2$-fellow-travellers. If $B$ is admissible then by \texttt{C4} we have that $\kappa_B$ precedes $\kappa_A$ in shortlex order so $B \prec A$. Thus, $B$ $2$-refutes $A$. On the other hand, if $B$ is not admissible then by Lemma \ref{lem:adjoingWhenInS} there is an element $C$ such that $|C|<|B|$, $\pi(C)=\pi(B)$, and $C$ and $B$ are $1$-fellow-travellers (since $C$ $1$-refutes $B$). This imply that $C$ $3$-refutes $A$ because (i) $|C|<|B|=|A|$ and thus $C \prec A$; (ii)  $\pi(C)=\pi(B)=\pi(A)$; (iii) $C$ and $A$ are $3$-fellow-travellers (this follows from Lemma \ref{lem:twoKFTs} since $A$ and $B$ are $2$-fellow-travellers and $B$ and $C$ are $1$-fellow-travellers).
\end{proof}

We complete the proof of Proposition \ref{prop:inefficientCanRefute} by proving the following proposition:

\begin{proposition} \label{prop:partOfPacingPair}
Let $A \in \Phi^*$ be an admissible but not efficient. Then, there is $B\in\Phi^*$ such that $(A,B)$ is a pacing pair.
\end{proposition}

The proof is broken into several lemmas. We begin by deriving some technical information in the situation where $\cord{\kappa_A}{\ell} = 1$.

\begin{lemma} \label{lem:leftExtNotToLong}
Let $A=\varphi_{W_1}\cdots\varphi_{W_n} \in \Phi^*$ be admissible. Suppose there is an index $1< \ell\leq n$ where $W_{\ell-1}$ has a decomposition $W_{\ell-1} = W_{\ell-1}' W_{\ell-1}''$ such that $W_{\ell-1}''W_\ell$ is in $\Sb$. Then, $\lp{W_{\ell-1}''} \leq 1$. Similarly, if $W_{\ell}$ has a decomposition $W_{\ell} = W_{\ell}' W_{\ell}''$ such that $W_{\ell-1} W_\ell'$ is in $\Sb$ then $\lp{W_{\ell}'} \leq 1$.
\end{lemma}
\begin{proof}
We prove the first case; the other case is similar. The lemma follows trivially if $W_{\ell-1}''$ is the empty word so assume $W_{\ell-1}'' \neq \eps$. Take $V_1,V_2$ and $U_1,U_2$ such that $V_1 W_{\ell-1}' W_{\ell-1}'' V_2 \in \R$ and $U_1 W_{\ell-1}'' W_\ell U_2 \in \R$. Since by admissibilty $W_{\ell-1}' W_{\ell-1}'' W_\ell$ is not in $\Sb$ we get that $U_1 \neq V_1 W_{\ell-1}'$. Hence, $W_{\ell-1}''$ is a piece and consequently, $\lp{W_{\ell-1}''} = 1$.
\end{proof}

\begin{lemma} \label{lem:PSizeAtleast3andMore}
Let $A=\varphi_{W_1}\cdots\varphi_{W_n} \in \Phi^*$ be admissible and suppose that $\cord{\kappa_A}{\ell}=1$ for some $1\leq\ell\leq n$. Let $W_{\ell-1} = W_{\ell-1}' W_{\ell-1}''$ and $W_{\ell+1} = W_{\ell+1}' W_{\ell+1}''$ be the decompositions
guaranteed by the definition of $\kappa_A$. Then, $\lp{W_{\ell-1}'' W_\ell W_{\ell+1}'} \geq 4$ and also $\lp{W_{\ell-1}''}, \lp{W_{\ell+1}'}\leq 1$.
\end{lemma}
\begin{proof}
By the definition of $\kappa_A$ and the $K_3^2$ condition we have that
\[
\lp{W_{\ell-1}'' W_\ell W_{\ell+1}'} > \lp{\Comp{W_{\ell-1}'' W_\ell W_{\ell+1}'}} \geq 3
\]
Hence, $\lp{W_{\ell-1}'' W_\ell W_{\ell+1}'} \geq 4$. By admissibility and Lemma \ref{lem:leftExtNotToLong} we have that $\lp{W_{\ell-1}''} \leq 1$ and that $\lp{W_{\ell+1}'} \leq 1$.
\end{proof}

\begin{lemma} \label{lem:uniquenessOfComp}
Let $A=\varphi_{W_1}\cdots\varphi_{W_n} \in \Phi^*$ be admissible and suppose that $\cord{\kappa_A}{\ell} = 1$ for some $1\leq\ell\leq n$. Then there are \emph{unique} decompositions $W_{\ell-1} = W_{\ell-1}' W_{\ell-1}''$ and $W_{\ell+1} = W_{\ell+1}' W_{\ell+1}''$ such that $W_{\ell-1}'' W_\ell W_{\ell+1}' \in \R$. 
\end{lemma}
\begin{proof}
It follows by Lemma \ref{lem:PSizeAtleast3andMore} that $\lp{W_{\ell-1}'' W_\ell W_{\ell+1}'} \geq 4$,
$\lp{W_{\ell-1}''} \leq 1$, and $\lp{W_{\ell+1}'}\leq1$. Now, $\lp{W_{\ell-1}''} + \lp{W_\ell} + \lp{W_{\ell+1}'} \geq \lp{W_{\ell-1}'' W_\ell W_{\ell+1}'}$ which show that $\lp{W_\ell}\geq 2$. Consequently, we have by Observation \ref{obs:twoPiecUniq} that there is a unique element $R\in\R$ that contains $W_\ell$ as a subword. This shows the uniqueness of the decompositions.
\end{proof}

Construction \ref{cnt:fixAtL} (below) is used in the inductive step of the proof of Proposition \ref{prop:partOfPacingPair}. 

\begin{construction}[Fixing $A$ at location $\ell$] \label{cnt:fixAtL}
Let $A=\varphi_{W_1}\cdots\varphi_{W_n} \in \Phi^*$ be admissible. Suppose that for some index $\ell$, where $1\leq\ell\leq n$, we have that $\cord{\kappa_A}{\ell}=1$. We construct an element $B=\varphi_{U_1}\cdots\varphi_{U_n}$ in the following way which we denote as ``fixing $A$ at location $\ell$''. It follows from Lemma \ref{lem:uniquenessOfComp} that since $\cord{\kappa_A}{\ell} = 1$ there are unique decompositions $W_{\ell-1} = W_{\ell-1}' W_{\ell-1}''$ and $W_{\ell+1} = W_{\ell+1}' W_{\ell+1}''$ such that $W_{\ell-1}'' W_\ell W_{\ell+1}' \in \R$ and $\lp{W_{\ell-1}'' W_\ell W_{\ell+1}'} > \lp{\Comp{W_{\ell-1}'' W_\ell W_{\ell+1}'}}$. Then, $B$ is defined by setting: $U_{\ell-1} = W_{\ell-1}'$, $U_{\ell} = \Comp{W_{\ell-1}'' W_\ell W_{\ell+1}'}$, $U_{\ell+1} = W_{\ell+1}''$, and $U_i=W_i$ for $i\neq\ell-1,\ell,\ell+1$. If $\ell=1$ or $\ell=n$ then $W_{0}$ or $W_{n+1}$, respectively, are undefined so we just ignore these indices.
\end{construction}

\begin{remark} \label{rem:AfterFixEqAndFT}
Suppose $B$ is constructed from $A$ by fixing $A$ at location $\ell$ (Construction \ref{cnt:fixAtL}). Here are some immediate consequences of the construction which are relevant to the definition of pacing pairs (the notation of Construction \ref{cnt:fixAtL} is used).
\begin{enumerate}[i)]
	\item \label{rem:AfterFixEqAndFT:eq} Clearly, $|A|=|B|$ and $\pi(A)=\pi(B)$.
	\item \label{rem:AfterFixEqAndFT:ft} We have $\pi(A(j)) = \pi(B(j))$ for all $1\leq j\leq n$ excluding $j=\ell-1$ and $j=\ell$ (recall that $A(j)$ is the prefix of $A$ of length $j$). This shows that $d_\Phi(A(j),B(j))=0$ for all $j<\ell-1$ and $j>\ell$. In addition, if $W_{\ell-1}''$ is not empty then we have the equality $\pi(B(\ell-1) \varphi_{(W_{\ell-1}'')}) = \pi(A(\ell-1))$ (if $W_{\ell-1}''$ is empty then we have the equality $\pi(B(\ell-1)) = \pi(A(\ell-1))$). Similarly,  if $W_{\ell+1}'$ is not empty then $\pi(A(\ell) \varphi_{W_{\ell+1}'}) = \pi(B(\ell))$. Therefore, $d_\Phi(A(j),B(j)) \leq 1$ for $j=\ell-1$ and $j=\ell$. Consequently, $A$ and $B$ are $1$-fellow-travellers.
	\item \label{rem:AfterFixEqAndFT:we} Following the details of the construction we have that $W_{\ell+1}$ is equal to $W_{\ell+1}' U_{\ell+1}$ and $U_j=W_j$ for all $j\geq\ell+2$. So, $U_{\ell+1}$ is a suffix of $W_{\ell+1}$. Also, $\lp{W_{\ell+1}'}\leq1$ by Lemma \ref{lem:leftExtNotToLong} so $\lp{W_{\ell+1}} \leq 1 + \lp{U_{\ell+1}}$ and we get that $\lp{U_{\ell+1}} \geq \lp{W_{\ell+1}} - 1$.
	\item \label{rem:AfterFixEqAndFT:cmp} We have $\lp{U_\ell} < \lp{\Comp{U_\ell}}$. This follows since $U_{\ell} = \Comp{W_{\ell-1}'' W_\ell W_{\ell+1}'}$ so $\Comp{U_\ell} = W_{\ell-1}'' W_\ell W_{\ell+1}'$. Thus, $\lp{U_\ell} < \lp{W_{\ell-1}'' W_\ell W_{\ell+1}'} = \lp{\Comp{U_\ell}}$.
\end{enumerate}
\end{remark}

After applying Construction \ref{cnt:fixAtL} to an inefficient element $A \in \Phi^*$ we get an element $B \in \Phi^*$ such that the pair $(A,B)$ is \emph{almost} a pacing pair. Specifically, out of the five conditions in the definition of a pacing pair (Definition \ref{def:pacingPair}) the first four conditions always hold. This is the content of the next lemma. Afterward, we give three special situation where the last condition (the fifth one) also hold (Lemma \ref{lem:trivCases}).

\begin{lemma} \label{lem:basisOfInd}
Suppose that $A \in \Phi^*$ is admissible and $\cord{\kappa_A}{\ell} = 1$ for some $1\leq\ell\leq |A|$. Suppose further that we construct $B$ by fixing $A$ at location $\ell$. The first four properties of an $\ell$-pacing pair hold for the pair $(A,B)$.
\end{lemma}
\begin{proof}
We check the first four conditions one by one (see Definition \ref{def:pacingPair}).
\begin{description}
	\item Condition \texttt{C1}: By our assumption $A$ is admissible and $\cord{\kappa_A}{\ell} = 1$.
	\item Condition \texttt{C2}: Follows from Remark \ref{rem:AfterFixEqAndFT} - part \ref{rem:AfterFixEqAndFT:eq}.
	\item Condition \texttt{C3}: Follows from Remark \ref{rem:AfterFixEqAndFT} - part \ref{rem:AfterFixEqAndFT:ft} by taking $j=\ell-1$ if $\ell>1$ or $j=\ell$ if $\ell=1$.
	\item Condition \texttt{C4}: Follows from Remark \ref{rem:AfterFixEqAndFT} - part \ref{rem:AfterFixEqAndFT:we}.
\end{description}
\end{proof}

\begin{lemma} \label{lem:contRedKappa}
Let $A \in \Phi^*$ be admissible and suppose that $\cord{\kappa_A}{\ell} = 1$ for some $1\leq\ell\leq |A|$. If we construct $B$ by fixing $A$ at location $\ell$ (Construction \ref{cnt:fixAtL}) then $\cord{\kappa_B}{\ell}=0$. In particular, $\cord{\kappa_B}{\ell} < \cord{\kappa_A}{\ell}$.
\end{lemma}
\begin{proof}
Suppose $B=\varphi_{U_1}\cdots\varphi_{U_n}$. Using the notation of Construction \ref{cnt:fixAtL} we have that $U_\ell\in\R$. If we assume by contradiction that $\cord{\kappa_B}{\ell}=1$ then there are unique decompositions $U_{\ell-1} = U_{\ell-1}' U_{\ell-1}''$ and $U_{\ell+1} = U_{\ell+1}' U_{\ell+1}''$ such that $U_{\ell-1}'' U_\ell U_{\ell+1}' \in \R$. Also, $\lp{U_{\ell-1}'' U_\ell U_{\ell+1}'} > \lp{\Comp{U_{\ell-1}'' U_\ell U_{\ell+1}'}}$. Since $U_\ell \in \R$ we get that $U_{\ell-1}'' U_\ell U_{\ell+1}' = U_\ell$ so $\lp{U_\ell} > \lp{\Comp{U_\ell}}$. This leads to a contradiction since $\lp{U_\ell} < \lp{ \Comp{ U_\ell } }$ by Remark \ref{rem:AfterFixEqAndFT} - part \ref{rem:AfterFixEqAndFT:cmp}.
\end{proof}

\begin{lemma} \label{lem:kappaRedAftrConstr}
Let $A \in \Phi^*$ be admissible and suppose that $\cord{\kappa_A}{\ell} = 1$ for some $1\leq\ell\leq |A|$. Suppose further that we construct $B$ by fixing $A$ at location $\ell$ (Construction \ref{cnt:fixAtL}) and $B$ is admissible. Then, $\cord{\kappa_B}{j} \leq \cord{\kappa_A}{j}$ for all $j < \ell-1$.
\end{lemma}
\begin{proof}
Let $A=\varphi_{W_1}\cdots\varphi_{W_n}$ and $B=\varphi_{U_1}\cdots\varphi_{U_n}$. Since $U_j = W_j$ for all $j < \ell-1$ we get that $\cord{\kappa_B}{j} \leq \cord{\kappa_A}{j}$ for all $j < \ell-2$ (actually there is an equality). This is also true for $j=\ell-2$ by the following reasons. We need to show that if $\cord{\kappa_B}{\ell-2}=1$ then also $\cord{\kappa_A}{\ell-2} = 1$. So assume that $\cord{\kappa_B}{\ell-2}=1$. By the definition of the auxiliary vectors, there are decompositions $U_{\ell-3} = U_{\ell-3}' U_{\ell-3}''$ and $U_{\ell-1} = U_{\ell-1}' U_{\ell-1}''$ such that $\lp{U_{\ell-3}'' U_{\ell-2} U_{\ell-1}'} > \lp{\Comp{U_{\ell-3}'' U_{\ell-2} U_{\ell-1}'}}$. But, $U_{\ell-1}$ is a prefix of $W_{\ell-1}$ and so there similar decompositions for $W_{\ell-3}$ and $W_{\ell-1}$ and thus $\cord{\kappa_A}{\ell-2}=1$.
\end{proof}

\begin{lemma} \label{lem:trivCases}
Suppose that $A \in \Phi^*$ is admissible and $\cord{\kappa_A}{\ell} = 1$ for some $1\leq\ell\leq |A|$. Suppose further that we construct $B$ by fixing $A$ at location $\ell$.
If we have that one of the following hold then $(A,B)$ is an $\ell$-pacing pair:
\begin{enumerate}[(a)]
	\item \label{lem:trivCases:bn} $B$ is not admissible.
	\item \label{lem:trivCases:l1} $\ell=1$.
	\item \label{lem:trivCases:ba} $B$ is admissible and $\cord{\kappa_B}{\ell-1} \leq \cord{\kappa_A}{\ell-1}$.
\end{enumerate}
\end{lemma}
\begin{proof}
In any of these instances the pair $(A,B)$ is a pacing pair since the last condition (\texttt{C5}) of Definition \ref{def:pacingPair} hold (the other conditions hold by Lemma \ref{lem:basisOfInd}). This follows trivially for (\ref{lem:trivCases:bn}), from Lemma \ref{lem:contRedKappa} for (\ref{lem:trivCases:l1}), and from Lemma \ref{lem:contRedKappa} and Lemma \ref{lem:kappaRedAftrConstr} for (\ref{lem:trivCases:ba}).
\end{proof}

We are now ready to complete the proof of Proposition \ref{prop:partOfPacingPair}

\begin{proof}[Proof of Proposition \ref{prop:partOfPacingPair}]
Suppose we are given some admissible element $A=\varphi_{W_1}\cdots\varphi_{W_n} \in \Phi^*$ and an index $\ell$ such that $\cord{\kappa_A}{\ell}=1$ (where, $1\leq\ell\leq|A|$). We show that there is an element $A=\varphi_{U_1}\cdots\varphi_{U_n} \in \Phi^*$ such that $(A,B)$ is an $\ell$-pacing pair. We prove by induction on $\ell$. If $\ell=1$ then we can construct $B$ by fixing $A$ at location $1$ (Construction \ref{cnt:fixAtL}) and by Lemma \ref{lem:trivCases} (\ref{lem:trivCases:l1}) we get that $(A,B)$ is a pacing pair.

Suppose $\ell>1$ and assume that the lemma holds for $\ell-1$. We construct $B$ by fixing $A$ at location $\ell$. If $B$ is not admissible we are done by Lemma \ref{lem:trivCases} (\ref{lem:trivCases:bn}). Hence we can assume that $B$ is admissible. By Lemma \ref{lem:contRedKappa} we have that $\cord{\kappa_B}{\ell} < \cord{\kappa_A}{\ell}$ and we are done if $\cord{\kappa_B}{\ell-1} \leq \cord{\kappa_A}{\ell-1}$ by Lemma \ref{lem:trivCases} (\ref{lem:trivCases:ba}). Hence, we may assume that $\cord{\kappa_B}{\ell-1} > \cord{\kappa_A}{\ell-1}$ so necessarily $\cord{\kappa_B}{\ell-1} = 1$ and $\cord{\kappa_A}{\ell-1} = 0$. Using induction hypothesis on $B$ (which is admissible and $\cord{\kappa_B}{\ell-1}=1$) we get an element $C=\varphi_{V_1}\cdots\varphi_{V_n}$ such that $(B,C)$ is an $(\ell-1)$-pacing pair. We claim that $(A,C)$ is an $\ell$-pacing pair. We verify this by checking the conditions of the definition of an $\ell$-pacing pair (Definition \ref{def:pacingPair}), one by one. 
\begin{description}
 \item \underline{Condition \texttt{C1}}. By assumption, $A$ is admissible and $\cord{\kappa_A}{\ell} = 1$.
 
 \item \underline{Condition \texttt{C2}}. Recall that we constructed $B$ by fixing $A$ at location $\ell$ and that $(B,C)$ is a pacing pair. Thus, we have the following two facts:
\begin{enumerate}[(i)]
	\item By Remark \ref{rem:AfterFixEqAndFT} - part \ref{rem:AfterFixEqAndFT:eq} we have the equalities $|A|=|B|$ and $\pi(A)=\pi(B)$; and,
	\item By condition \texttt{C2} of Definition \ref{def:pacingPair} we have the equalities $|B|=|C|$ and $\pi(B)=\pi(C)$.
\end{enumerate}
Consequently, $|A|=|C|$ and $\pi(A)=\pi(C)$.

 \item \underline{Condition \texttt{C3}}. $(B,C)$ is an $(\ell-1)$-pacing pair and thus by condition \texttt{C3} there is some index $1\leq j\leq\ell-1$ such that:
 \begin{enumerate}
  \item $d_\Phi(B(i),C(i))=0$ for all $i<j$ and $i>\ell-1$.
  \item $d_\Phi(B(\ell-1),C(\ell-1))\leq1$.
  \item $d_\Phi(B(i),C(i))\leq2$ for all $j\leq i\leq\ell-2$.
 \end{enumerate}
$B$ was constructed by fixing $A$ at location $\ell$ hence $d_\Phi(A(i),B(i)) = 0$ for all $i<\ell-1$ and $i>\ell$ (Remark \ref{rem:AfterFixEqAndFT} - part \ref{rem:AfterFixEqAndFT:ft}). It follows therefore that $d_\Phi(A(i),C(i)) \leq d_\Phi(A(i),B(i)) +  d_\Phi(B(i),C(i)) \leq 0 + 2 = 2$ for all $j\leq i\leq\ell-2$. Again by Remark \ref{rem:AfterFixEqAndFT} - part \ref{rem:AfterFixEqAndFT:ft}, we have $d_\Phi(A(\ell-1),B(\ell-1)) \leq 1$ which implies that $d_\Phi(A(\ell-1),C(\ell-1)) \leq d_\Phi(A(\ell-1),B(\ell-1)) +  d_\Phi(B(\ell-1),C(\ell-1)) \leq 1 + 1 = 2$. Also, using the remark again, we have $d_\Phi(A(\ell),C(\ell)) \leq d_\Phi(A(\ell),B(\ell)) + d_\Phi(B(\ell),C(\ell)) \leq 1 + 0 = 1$. Finally, by similar consideration, we have that $d_\Phi(A(i),C(i)) \leq d_\Phi(A(i),B(i)) +  d_\Phi(B(i),C(i)) \leq 0 + 0 = 0$ for all $i<j$ and $i>\ell$.

 \item \underline{Condition \texttt{C4}}. For $j\geq \ell+2$ we have that $W_j=U_j=V_j$ (using the induction hypothesis and the fact that $B$ was constructed from $A$ by fixing it at location $\ell$). Moreover, $V_{\ell+1} =  U_{\ell+1}$ and also $U_{\ell+1}$ is a subword of $W_{\ell+1}$ such that $\lp{U_{\ell+1}} \geq \lp{W_{\ell+1}} - 1$ (Remark \ref{rem:AfterFixEqAndFT} - part \ref{rem:AfterFixEqAndFT:we}) and thus $V_{\ell+1}$ is a subword of $W_{\ell+1}$ such that $\lp{V_{\ell+1}} \geq \lp{W_{\ell+1}} - 1$.

 \item \underline{Condition \texttt{C5}}. Assume $C$ is admissible. Since $(B,C)$ is an $(\ell-1)$-pacing pair we have that $\cord{\kappa_C}{\ell-1} < \cord{\kappa_B}{\ell-1}$ and $\cord{\kappa_C}{i} \leq \cord{\kappa_B}{i}$ for all $1\leq i \leq \ell-2$. Hence, $\cord{\kappa_C}{\ell-1} = 0 = \cord{\kappa_A}{\ell-1}$ (last equality follow from our assumptions above). By Lemma \ref{lem:kappaRedAftrConstr} we have that $\cord{\kappa_B}{j} \leq \cord{\kappa_A}{j}$ for all $j < \ell-1$. Thus, $\cord{\kappa_C}{j} \leq \cord{\kappa_A}{j}$ for all $j < \ell-1$. To finish, we need to verify that $\cord{\kappa_C}{\ell} < \cord{\kappa_A}{\ell}$. Since $\cord{\kappa_A}{\ell} = 1$ we need to show that $\cord{\kappa_C}{\ell} = 0$. $(B,C)$ is an $(\ell-1)$-pacing pair hence $V_\ell$ is a subword of $U_\ell$ with $\lp{V_{\ell}} \geq \lp{U_{\ell}} - 1$. Now, $B$ was constructed from $A$ by fixing $A$ at location $\ell$ so $U_{\ell}$ is an element of $\R$ and thus $\lp{U_{\ell}} \geq 3$. This shows that $\lp{V_{\ell}} \geq 2$. Consequently, using Observation \ref{obs:twoPiecUniq}, $U_\ell$ is the only element in $\R$ such that $V_\ell$ is its subword. This completes the argument since by part \ref{rem:AfterFixEqAndFT:cmp} of Remark \ref{rem:AfterFixEqAndFT} we have that $\lp{U_\ell} < \lp{\Comp{U_\ell}}$ so necessarily $\cord{\kappa_C}{\ell} = 0$.

\end{description}
\end{proof}

\end{document}